\documentclass[a4paper,12pt]{amsart}
\input{macros_2}
\usepackage[margin=.99in]{geometry}
\setlist{ 
	listparindent=\parindent,
	parsep=0pt,
}

\numberwithin{equation}{section}

    \author[Jefferson Baudin]{Jefferson Baudin}
  \address{\'Ecole Polytechnique F\'ed\'erale de Lausanne, Chair of Algebraic Geometry \newline 
    \indent MA C3 575 (Bâtiment MA), Station 8, CH-1015 Lausanne}
  \email{jefferson.baudin@epfl.ch}

  \author[Sofia Tirabassi]{Sofia Tirabassi}
  \address{Department of Mathematics, Stockholm University\newline
  \indent SE-106 91 Stockholm, Sweden}
  \email{tirabassi@math.su.se}


\title[Effective characterization of semi-abelian varieties]{Effective characterization of semi-abelian varieties}
\begin{document}
	
\begin{abstract}
    We obtain effective characterizations of complex semi-abelian varieties in arbitrary dimension in terms of logarithmic irregularity and the first two logarithmic plurigenera, among varieties of maximal Albanese dimension, or among varieties whose compactification has large irregularity (these hypotheses are sharp). To the authors' knowledge, this is the first result in this direction for quasi-projective varieties in arbitrary dimension.
    
    %Our first result is a characterization among We have two main results: the first one is a characterization of semi-abelian varieties among quasi-projective varieties of maximal Albanese dimension. The second one is a characterization of semi-abelian varieties among quasi-projective varieties with large logarithmic irregularity and whose compactification has large irregularity. Both results are sharp.
\end{abstract}	

    \maketitle
	\tableofcontents
\section{Introduction}

A celebrated theorem of Chen and Hacon \cite{Chen_Hacon_Characterization_of_abelian_varieties} asserts that a smooth projective complex variety $X$ is birational to an abelian variety if and only if its irregularity satisfies $q(X) = \dim(X)$ and its first two plurigenera satisfy $P_1(X) = P_2(X) = 1$ (see also \cite{Pareschi_Basic_results_on_irr_vars_via_FM_methods, Jiang_An_effective_version_of_a_thm_of_Kawamata_on_the_Albanese_map} for other proofs and generalizations, and  \cite{Baudin_Effective_Characterization_of_ordinary_abelian_varieties, Baudin_Martin_Enriques_characterization_of_abelian_surfaces_in_positive_characteristic} in positive characteristic). The objective of this paper is to extend these results to the quasi-projective setting.

%This paper extends to the quasi-projective setting well known and celebrated characterization theorems for abelian varieties. It is the first such characterization for varieties of dimension greater than two, and generalizes to higher dimension the results in \cite{MLPT2023}.

Given a smooth complex quasi-projective variety $V$ and a smooth compactification $X$ with simple normal crossing (snc) boundary divisor $D$, we define the following logarithmic cohomological invariants of $V$ (which do not depend on the log smooth compactification):
\begin{itemize}
    \item the logarithmic (log) irregularity of $V$ is $\overline{q}(V) \coloneqq h^0(X,\Omega_X^1(\log D))$;
    \item given $m \geq 1$, the $m$-th logarithmic plurigenus of $V$ is $\overline{P}_m(V)\coloneqq h^0(X, \omega_X(D)^{\otimes m
    })$;
    \item the log-Kodaira dimension is $\overline{\kappa}(V) = \kappa(X, \omega_X(D))$.
\end{itemize}
Recall that a semi-abelian variety is a connected algebraic group which is the extension of an abelian variety by an affine torus (equivalently, a connected algebraic group that does not contain any copy of $\bG_a$). In analogy with the projective setting, Fujino showed in \cite{Fujino_On_quasi_Albanese_maps} that integrating logarithmic $1$-forms gives rise to a universal morphism $\alb_V \colon V \to \Alb(V)$ to a semi-abelian variety, which we again call the \emph{Albanese morphism}. We say that $V$ has \emph{maximal Albanese dimension} if $\dim(\alb_V(V)) = \dim(V)$. Our objective is to study the following conjecture:

\begin{conjecture*}[{\cite{MLPT2023}}]\label{general_conjecture}
    Let $d \geq 1$ be an integer. Then there exists an integer $k = k(d) \geq 1$ such that for any smooth complex variety $V$ of dimension $d$, if $\overline{q}(V) = \dim(V)$ and $\overline{P}_1(V) = \dots = \overline{P}_k(V) = 1$, then $\alb_V \colon V \to \Alb(V)$ is an isomorphism away from a closed subset of $\Alb(V)$ of codimension at least $2$.
\end{conjecture*}

In Iitaka's language \cite{Iitaka1975}, the above conclusion can be rephrased by saying that $\alb_V$ is a WWPB equivalence\footnote{As opposed to the projective situation, logarithmic invariants are not birational invariants since we can remove divisors. However, thanks to the work of Iitaka, they are invariant under WWPB equivalence. For this reason, it is commonly accepted that WWPB equivalence is the correct equivalence relation to consider when investigating the birational geometry of quasi-projective varieties.}. The first result in this direction was a theorem of Kawamata \cite{Kawamata_Characterization_of_abelian_varieties} using semipositivity methods, obtaining birationality of $\alb_V$ under the additional assumption that $\overline{\kappa}(V) = 0$ (i.e. fixing all plurigenera for $m$ sufficiently large and divisible). The birationality statement was later improved to the actual conclusion of \autoref{general_conjecture} in \cite{Fujino2024Erratum}.

Let us now move to effective results. First, note that the main result of \cite{Chen_Hacon_Characterization_of_abelian_varieties} states that \autoref{general_conjecture} holds with $k = 2$ in any dimension for projective varieties. However, it was shown in \cite{MLPT2023} that there exists a smooth surface $V$ with $\overline{q}(V) = 2$, $\overline{P}_1(V) = \overline{P}_2(V) = 1$, such that $\alb_V$ is not dominant. Hence, we truly need to fix higher logarithmic plurigenera in \autoref{general_conjecture}. The authors in \cite{MLPT2023} settled this conjecture in dimension $d = 2$, showing that the bound $k = 3$ holds.

In this article, we started a systematic study of this conjecture in arbitrary dimension. In order to clarify the picture, we wanted to understand under what assumptions the usual bounds $\overline{P}_1(V) = \overline{P}_2(V) = 1$ are enough. Our first result is the following:

\begin{theorem}\label{main_thm_max_alb_dim}
	Let $V$ be a smooth complex variety of maximal Albanese dimension such that $\overline{P}_2(V) = 1$. Then $a_V \colon V \to \Alb(V)$ is an isomorphism away from a closed subset of $\Alb(V)$ of codimension at least $2$.
\end{theorem}

Thanks to this result, we see that it is enough to show that $\alb_V$ is dominant in \autoref{general_conjecture}. It is therefore interesting to see that the hardest part of \autoref{general_conjecture} seems to show dominance, while in the projective case this was an immediate consequence of the generic vanishing theory developed by Green--Lazarsfeld \cite{Green_Lazarsfeld_Deformation_theory_generic_vanishing_theorems_and_some_conjectures_of_Enriques_Catanese_Beauville, Green_Lazarsfeld_Higher_obstruction_to_deforming_cohomology_groups_of_line_bundles}, and was explicitly proven in \cite{Ein_Lazarsfeld_Singularities_of_theta_divisors_and_the_birational_geometry_of_irregular_varieties} (see also \autoref{lem:GVmorphism}).

As a corollary of our result, we obtain the following characterization of affine tori:

\begin{theorem}\label{main_thm_intro_affine_case}
    A smooth affine complex variety of maximal Albanese dimension $V$ is isomorphic to $\mathbb{G}_m^{\dim(V)}$ if and only if $\overline{P}_2(V) = 1$.
\end{theorem}

In \cite{MLPT2023}, the authors also showed that if $V$ is a smooth surface with $\overline{q}(V) = 2$, $\overline{P}_1(V) = \overline{P}_2(V) = 1$ and $q(X) = 1$ (for a compactification $X$ of $V$), then $\alb_V \colon V \to \Alb(V)$ also satisfies the conclusion of \autoref{general_conjecture} (their surface $V$ violating the bound $k = 2$ satisfies $\Alb(V) = \bG_m^2$, i.e. $q(X) = 0$). We show a similar result in higher dimension:

\begin{theorem}\label{main_thm}
    Let $V$ a smooth complex variety with $\overline{q}(V) = \dim V$ and $\overline{P}_1(V) = \overline{P}_2(V) = 1$. Assume furthermore that $q(X) \geq \dim(V) - 1$ for a smooth compactification $X$ of $V$. Then $\alb_V \colon V \to \Alb(V)$ is an isomorphism away from a closed subset of $\Alb(V)$ of codimension at least $2$.
\end{theorem} 

As already mentioned, this assumption on the irregularity of a compactification is sharp. \\

Our paper is organized as follows:
in \autoref{sec:BG} we recall the necessary background material about semi-abelian varieties and generic vanishing, and we prove a useful theorem about higher direct images of logarithmic canonical bundles. We also prove a few preliminary results.

The proofs of \autoref{main_thm_max_alb_dim} and \autoref{main_thm_intro_affine_case} are the content of \autoref{sec:maxAlb}. Rather quick arguments show that our assumptions warrant that $\overline{q}(V)=\dim V$, and that the Albanese morphism of a compactification of $V$ is surjective with connected fibers. The main part of this section is in fact devoted to showing that $\alb_V$ is birational, and the idea is essentially to show that $\alb_V$ admits a compactification which is étale over $\Alb(V)$ (up to exceptional components). Since a finite étale cover of a semi-abelian variety is again a semi-abelian variety, this concludes the statement by the universal property of the Albanese morphism. Showing this étale property follows from a meticulous study of the (log) ramification divisor. Finally, the conclusion that $\alb_V$ is an isomorphism away from a closed subset of $\Alb(V)$ of codimension at least $2$ follows via a brief argument using our previous considerations.

Finally, we prove \autoref{main_thm} in \autoref{sec:largeq}. We first obtain a general lower bound for the dimension of $a_V(V)$ that gives dominance under our conditions (and hence the result by \autoref{main_thm_max_alb_dim}), under the assumption that the Albanese morphism of a compactification $X$ is surjective with connected fibers. The rest of this section is therefore devoted to showing this latter result under our hypotheses, see \autoref{thm:fiberspace}. This can be seen as a partial extension of \cite[Theorem 3.1]{Jiang_An_effective_version_of_a_thm_of_Kawamata_on_the_Albanese_map} for log canonical pairs. The argument for this final result is an adaptation of the more general approach taken in \cite[Theorem 3.4.5]{Baudin_Effective_Characterization_of_ordinary_abelian_varieties}, in positive characteristic.

\subsection{Acknowledgements}
We are very grateful to Sung Gi Park for carefully explaining us ideas of the proof of \autoref{hodge_module_lemma}. We would also like to thank Raymond Cheng, Yajnaseni Dutta, Stefano Filipazzi, Osamu Fujino, Christina Kapatsori, Fanjun Meng, Léo Navarro Chafloque, Zsolt Patakfalvi, Dan Petersen, Mihnea Popa, Roberto Svaldi and Nikolaos Tsakanikas for interesting discussions related to the content of this article. The first author was funded by grant \#804334 from the European Research
Council (ERC) and by grant \#20021/231484 from the
Swiss National Science Foundation. The second author was partially supported by the grant number VR2023-03837 of the \emph{Vetenskaprådet} and by the grant number 2023.0395 of the Knut and Alice Wallenberg foundation. 

\subsection{Notations}

\begin{itemize}
    \item A \emph{variety} is an integral quasi--projective scheme of finite type over $\bC$.
    \item A pair $(X, D)$ consisting of a projective variety $X$ and a reduced divisor $D$ is \emph{log smooth} if $X$ is smooth and $D$ has simple normal crossings (in short, snc).
    \item Given a smooth variety $V$, we say that a log smooth pair $(X, D)$ is a \emph{good compactification} of $V$ if $X$ admits an open embedding $V \hookrightarrow X$ whose complement is $D$.
    %\item We use the notation of pairs and quasi-projective varieties interchangeably.  For example, by a map $(X,D)\rightarrow (Y,\Delta)$ we mean a morphism $X\backslash D\rightarrow Y\backslash \Delta$. \JB{mmmh really?}
    \item An algebraic fiber space (or fibration) is a proper morphism $f \colon X \to Y$ of normal varieties such that $f_*\cO_X = \cO_Y$, that is, a proper morphism with connected fibers. 
    \item Given a smooth variety $V$, we set $\omega_V \coloneqq \det(\Omega^1_V)$. More generally, given a separated, Cohen--Macaulay scheme of finite type over $\bC$ (e.g. a simple normal crossing divisor on a smooth variety) with structural map $\pi \colon V \to \Spec(\bC)$, we set $\omega_V$ to be $\pi^!\cO_{\Spec \bC}[-\dim(V)]$ (see \stacksproj{0A9Y}), which is a sheaf by \stacksproj{0AWS}.
\end{itemize}

\section{Background and preliminary results}\label{sec:BG}

\subsection{Semi-abelian varieties and Albanese maps}\label{semi-alb}
We recall that a \emph{semi-abelian variety} (often found in literature under the name of \emph{quasi-abelian variety}) is an algebraic group $G$ arising as an extension of an abelian variety $A$ by an algebraic torus $\mathbb{G}_m^r$. The abelian variety $A$ above is said to be the \emph{compact part} of $G$, and $\bG_m^r = \ker(G \to A)$ is called the \emph{torus part} of $G$.

\begin{lemma}[{\cite[Proof of Lemma 3.8]{Fujino_On_quasi_Albanese_maps}}]\label{standard_compactification}
    Let $G$ be a semi-abelian variety with compact part $A$ and torus part $\bG_m^r$. Then there exists a good compactification $(Z, \Delta)$ of $G$ such that the induced map $\pi \colon Z \to A$ is a $(\bP^1)^r$-bundle, and such that $\Omega_Z^1(\log \Delta) \cong \cO_Z^{\oplus \dim(Z)}$.
\end{lemma}
\begin{proof}
     Recall that a global frame of a vector bundle means a set of global sections that forms a basis at each fiber. Let $\bG_m^r$ act on $G$ by translation, and on $(\bP^1)^r$ by the usual action \[(t_1, \dots, t_r) \cdot ([x_1; y_1], \dots, [x_r; y_r]) = ([t_1 x_1; y_1], \dots, [t_r x_r; y_r]).\] We can then define $(Z, \Delta)$ to be the diagonal quotient of $G \times (\bP^1, 0 + \infty)^r$ by $\bG_m^r$. By construction, the induced map $Z \to A$ is a $(\bP^1)^r$-bundle. Note that each logarithmic one-form $\dlog(x_i)$ is $\bG_m^r$-invariant, and therefore they form a global frame of $\Omega^1_{Z/A}(\log \Delta)$ (this can be checked on fibers of $\pi$). Pulling back a global frame of $\Omega_A^1$ therefore induces a global frame of $\Omega^1_Z(\log \Delta)$.
\end{proof}

There are other possible compactification that the one of \autoref{standard_compactification}, e.g. one can compactify $G$ as a $\mathbb{P}^r$-bundle. However, $Z$ is the compactification that we will use throughout this paper, and we will refer to it as the \emph{standard compactification of $G$}.

For any smooth complex quasi-projective variety $V$, there is a pair $(\Alb(V),a_V)$, consisting of a semi-abelian variety $\Alb(V)$ and a morphism $a_V\colon V\rightarrow \Alb(V)$ such that any morphism $f\colon V\rightarrow G$ to a semi-abelian variety $G$ factors (uniquely) through a homomorphism of algebraic groups $\Alb(V)\rightarrow G$. The pair $(\Alb(V),a_V)$ is unique up to unique isomorphism (and up to a choice of a base point in $V(\mathbb{C})$). We say that $\Alb(V)$ is the \emph{(semi) Albanese variety of $V$} and $a_V\colon V\rightarrow \Alb(V)$ is the \emph{(semi)Albanese morphism}.

It follows by construction (see \cite{Fujino_On_quasi_Albanese_maps}) that $\dim A(V)=\overline{q}(V)$, and that the compact part of $A(V)$ is simply the Albanese variety $A(X)$ of (any) good compactification $(X, D)$ of $V$. In addition, given a good compactification $(X, D)$ of $V$, we have the following useful commutative diagram (up to blowing up $X$ along the boundary):
\begin{equation}\label{eq:setupdiagram}
\begin{tikzcd}
    V \arrow[rr, hook] \arrow[dd, "a_V"'] &                 & X \arrow[dd, "g"] \arrow[lddd, "a_X", bend left=60] &  & \\
                    & & & & \\
    A(V) \arrow[rr, hook] \arrow[rd]      &                 & Z \arrow[ld, "a_Z"]                                 &  & {}  \\
    & A(X)\cong A(Z) & & &   
\end{tikzcd}
\end{equation}
where $Z$ is the standard compactification of $A(V)$. Throughout, if there is no chance of confusion, we will denote $a_X$ as $a$, while we will keep the subscript for $a_V$.

We recall the following crucial definition:
\begin{definition}
    We say that a smooth variety $V$ is of \emph{maximal Albanese dimension} if $\dim(V)=\dim(\alb_V(V))$.
\end{definition}
%In what follow we will need this easy lemma.
%\begin{lemma}\label{lem:Albanese_open_immersion}
%    Let $U \inc V$ be an open immersion of smooth varieties. Then the natural map $A(U) \to A(V)$ is surjective.
%\end{lemma}
%\begin{proof}
%    This follows immediately from the universal property, which implies that $V$ generates $A(V)$ as an algebraic group.
%\end{proof}

%\JB{did we ever use it?}\ST{It was used in the topologicial argument}

We conclude this paragraph with a useful isotriviality result about connected abelian Lie groups (such as semi--abelian varieties). This is surely well--known, but we could not find a reference.

\begin{lemma}\label{sequence_Lie_groups_split}
    Let $f \colon G \to H$ be a surjective morphism of connected abelian Lie groups, with $\ker(f)$ connected. Then there exists a homeomorphism $G \cong H \times \ker(f)$ over $H$.
\end{lemma}
\begin{proof}
    By assumption, we have a commutative diagram \[ \begin{tikzcd}
        0 \arrow[rr] &  & \pi_1(G) \arrow[rr] \arrow[d, two heads, "\pi_1(f)"] &  & \Lie(G) \arrow[rr] \arrow[d, two heads, "\Lie(f)"] &  & G \arrow[rr] \arrow[d, two heads, "f"] &  & 0 \\
        0 \arrow[rr] &  & \pi_1(H) \arrow[rr]                      &  & \Lie(H) \arrow[rr]                      &  & H \arrow[rr]                      &  & 0.
    \end{tikzcd} \]  where the morphism $\pi_1(G) \to \pi_1(H)$ is surjective by the associated long exact sequence in homotopy groups associated to the fibration $G \to H$ (in the sense of topology) and the fact that $\ker(f)$ is connected. Thus, all we need to find is a splitting of $\Lie(G) \to \Lie(H)$ inducing a splitting of $\pi_1(G) \to \pi_1(H)$. Take a $\bZ$--basis $\{v_1, \dots, v_s, v_{s + 1}, \dots, v_n\}$ of $\pi_1(G)$ such that $\{v_{s + 1}, \dots, v_n\}$ is a $\bZ$--basis of $\ker(\pi_1(f))$. Extend $\{\pi_1(f)(v_1), \dots, \pi_1(f)(v_s)\}$ into a $\bR$--basis $\{\pi_1(f)(v_1), \dots, \pi_1(f)(v_s), \pi_1(f)(w_1), \dots, \pi_1(f)(w_r)\}$ of $\Lie(H)$. Sending each $\pi_1(f)(v_i)$ to $v_i$ with $1 \leq i \leq s$ (resp. $\pi_1(f)(w_j)$ to $w_j$ with $1 \leq j \leq r$) gives the required splitting.
\end{proof}

\subsection{Generic vanishing}
Let $A$ be an abelian variety of dimension $g$ with dual $\bighat{A}$ and normalized Poincaré bundle $\mathscr{P}$. Given a coherent sheaf $\mathscr{F}$ on $A$ and an integer $i \geq 0$, we define the cohomological support loci of $\mathscr{F}$ as
\[ V^i(A, \mathscr{F})\coloneqq \set{\alpha \in \Pic^0(A)}{H^i(A, \mathscr{F}\otimes\alpha)\neq 0} \inc \bighat{A}.\]
We say that $\mathscr{F}$ is a GV-sheaf if $\codim V^i(A,\mathscr{F})\geq i$ for every $i\geq 0$. If this is the case, then by \cite[Corollary 3.5]{Hacon_A_derived_category_approach_to_generic_vanishing} we have a chain of inclusions
\[V^0(A,\mathscr{F})\supseteq V^1(A,\mathscr{F})\supseteq\cdots\supseteq V^g(A,\mathscr{F}).\]
A key tool to extract geometric information from the cohomological support loci is the \emph{symmetric Fourier--Mukai transform}. This functor, denoted by
\[\FM_A\colon D^b(A)^{op}\longrightarrow D^b(\bighat A),\] is obtained by composing the Grothendieck dual $\cR\HHom(\star, \cO_A)[g]$ with the integral transform with kernel $\mathscr{P}$ on $A\times\bighat A$, namely the functor $Rp_{\hat A, *}(p_A^*(\star)\otimes^L\mathscr{P})$. A celebrated result of Mukai \cite{Mukai_Duality_between_DX_and_D_hat_X} (see also \cite{Schnell_Fourier_Mukai_transform_made_easy}) asserts that this functor is an equivalence of triangulated categories, satisfying
\begin{equation}\label{eq:identity}
    \FM_A\circ\FM_{\hat A}\cong\id_{\bighat A}, \quad\FM_{\bighat A}\circ\FM_{ A}\cong \id_{A}
\end{equation} after identifying $A$ with its double dual, and the formula 
\begin{equation}\label{FM_and_cohom}
     H^i(A, \mathscr{F} \otimes \alpha)^{\vee} \cong \Tor_i(\FM_A(\mathscr{F}), \bC(\alpha^{\vee}))
\end{equation} for all $\alpha \in \Pic^0(A)$, $i \geq 0$ and $\mathscr{F} \in \Coh(A)$ \cite[Propositions 4.1 and 5.1]{Schnell_Fourier_Mukai_transform_made_easy}. Much else could be said, which goes beyond the scope of this article. An important result for us is the following theorem of Pareschi--Popa (see also \cite{Hacon_A_derived_category_approach_to_generic_vanishing}):
\begin{thm}[{\cite[Theorem 2.2]{PareschiPopa2009}}]\label{thm:equivalent GV} The following are equivalent for a coherent sheaf $\mathscr{F}$ on an abelian variety $A$:
\begin{itemize}
    \item the sheaf $\mathscr{F}$ is a GV-sheaf;
    \item the complex $\operatorname{FM}_A\left(\mathscr{F}\right)$ is concentrated in degree zero (in this case, its support is precisely $V^0(A,\mathscr{F})$).
\end{itemize}
\end{thm}

The following standard result will be our starting point:

\begin{proposition}\label{lem:GVmorphism}
   Let $(X,D)$ be a log smooth pair, and consider a morphism $b\colon X\rightarrow A$ to an abelian variety of dimension $g$. Then the following statements hold:
\begin{enumerate}
    \item\label{logGV} the sheaf $b_*\omega_X(D)$ is a GV-sheaf and its cohomology support loci are union of torsion translates of subtori of $\bighat A$.
    \end{enumerate}
    Suppose in addition that 
    \[ h^0(X,\omega_X(D))=h^0(X,\omega_X(D)^{\otimes 2})=1 \]
    and that $b^* \colon \Pic^0(A) \to \Pic^0(X)$ is injective (e.g. $b$ is the Albanese morphism). Then
    \begin{enumerate}[start=2]
        \item\label{lem:noZ} there is no positive-dimensional closed subvariety $Z\subseteq V^0(A,b_*\omega_X(D))$ such that $-Z\subseteq V^0(A, b_*\omega_X(D))$. In particular $\mathcal{O}_A$ is an isolated point of $V^0(A,b_*\omega_X(D))$.
        \item\label{lem:GVmorphismitem}  The sheaf $b_*\omega_X(D)$ has a direct summand which is a non-zero unipotent vector bundle. In particular, there exists a surjective morphism $b_*\omega_X(D)\rightarrow \mathcal{O}_A$ and $H^g(A, b_*\omega_X(D)) \neq 0$, and $b$ is surjective.
    \end{enumerate}
\end{proposition}
\begin{proof}
The sheaf $b_*\omega_X(D)$ is a GV-sheaf by \cite[Variant 5.5]{popa2014direct}, and the components of $V^0(X,a_*\omega_X(D))$ are torsion translates of subtori by \cite[Theorem 1.3]{shibata2016generic}. Item \autoref{lem:noZ} is standard, and a full proof can for example be found in \cite[Lemma 1.19]{MLPT2023}.

We are left with proving Item \autoref{lem:GVmorphismitem}. Since $\cO_A$ is an isolated point of $V^0(A, b_*\omega_X(D))$, we know by combining \autoref{logGV} and \autoref{thm:equivalent GV} that the origin is an isolated point of the support of $\FM_A(b_*\omega_X(D))$. In particular we can write 
    \[ \FM_A(b_*\omega_X(D)) = \mathscr{G}\oplus\mathscr{G'} \] where $\mathscr{G}$ is supported on $\mathcal{O}_A$. Applying $\FM_{\bighat{A}}$ and using \autoref{eq:identity} yields that $b_*\omega_X(D)$ contains $\FM_{\bighat{A}}(\mathscr{G})$ as a direct summand. Since this sheaf is a unipotent vector bundle by \cite[Example 2.9]{Mukai_Duality_between_DX_and_D_hat_X}, the statement is proven.
\end{proof}

\subsection{Higher direct images of logarithmic canonical sheaves}

Our objective here is to generalize a result of Kollár about higher direct images of canonical sheaves to logarithmic sheaves \cite{Kollar_Higher_direct_images_of_dualizing_sheaves_I, Kollar_Higher_direct_images_of_dualizing_sheaves_II} in a certain case (see \autoref{hodge_module_lemma}). We warmly thank Sung Gi Park for carefully explaining us ideas around the proof of this result. Throughout, we will extensively use the language of filtered right $\cD$--modules and Hodge modules \cite{Saito_Modules_de_Hodge_polarisables, Saito_Mixed_Hodge_modules}. The direct image functor on $\cD_X$--modules along a morphism $f \colon X \to Y$ will be denoted $f_+$ in order not to be confused with the functor $Rf_*$ on quasi--coherent sheaves. All filtrations will be denoted by $F$, and given a non--zero filtered $\cD$--module $\cM$, we will denote by $p(\cM) \in \bZ$ the smallest integer such that $F_{p(\cM)}(\cM) \neq 0$. Given a smooth variety $X$ of dimension $n$, the $\cD_X$--module $\omega_X$ will be given its canonical filtration (namely $F_{i}(\omega_X) = 0$ for $i < -n$, and $F_{i}(\omega_X) = \omega_X$ for all $i \geq -n$).

\begin{lemma}\label{preparation_lemma_HM}
    Let $X$ be a normal variety of dimension $n$, let $D$ be a Cartier divisor, and set $U \coloneqq X \setminus D$ with corresponding open immersion $j \colon U \to X$. Then the following statements hold:
    \begin{enumerate}
    \item\label{itm:exactness_open_immersion} the functor $j_+$ is exact on $\cD_U$--modules;
    \end{enumerate}
    If $(X, D)$ is in addition log smooth, then:
    \begin{enumerate}[start=2]
        \item\label{itm:filtration_on_pushforward} given a filtered $\cD_U$--module $\cM$ underlying a mixed Hodge module, we have that $p(j_+\cM) \geq p(\cM)$;
        \item\label{itm:case_of_trivial_HM} we have $F_{-n}(j_+\omega_U) = \omega_X(D)$.
    \end{enumerate}
\end{lemma}
\begin{proof}
    Since $j$ is an open immersion, it holds by construction that for any $\cD_U$--module $\cM$, the underlying quasi--coherent sheaf of $j_+\cM$ is $Rj_*\cM$. Since $j$ is exact, \autoref{itm:exactness_open_immersion} follows. To see \autoref{itm:filtration_on_pushforward}, note that $F_p(j_+\cM)|_U = F_p(\cM)$ for all $p \in \bZ$ by the discussion neighbouring \cite[Theorem 30.1]{Schnell_An_overview_of_Saito_theory_of_mixed_Hodge_modules}. Since $j_+\cM = j_*\cM$ as a quasi--coherent sheaf, no non--zero piece of the filtration of $F_p(j_+\cM)$ can be annihilated by restricting to $U$, so \autoref{itm:filtration_on_pushforward} is proven. Finally, \autoref{itm:case_of_trivial_HM} follows from \cite[Proposition 10.1]{Mustata_Popa_Hodge_ideals} (see Definition 9.4 in \emph{loc. cit.}).
\end{proof}

\begin{thm}\label{hodge_module_lemma}
    Let $f \colon X \to Y$ be a projective and surjective morphism of complex varieties of relative dimension $s$, with $X$ smooth. Let $\Delta \inc Y$ be a Cartier divisor, and assume that $D \coloneqq f^{-1}(\Delta)$ is snc. Then each sheaf $R^if_*\omega_X(D)$ is torsion--free (in particular $R^if_*\omega_X(D) = 0$ for all $i > s$), and there is an isomorphism \[ Rf_*\omega_X(D) \cong \bigoplus_{i \in \bZ} R^if_*\omega_X(D)[-i]. \] If in addition $Y$ is smooth, $f_*\cO_X = \cO_Y$ and $\Delta$ is also snc, then $R^sf_*\omega_X(D) \cong \omega_Y(\Delta)$.
\end{thm}
\begin{remark}
    \begin{enumerate}
        The above result is in particular valid with $D = \Delta = 0$, in which case it is due to Koll\'ar \cite{Kollar_Higher_direct_images_of_dualizing_sheaves_I, Kollar_Higher_direct_images_of_dualizing_sheaves_II}.
    \end{enumerate}
\end{remark}
\begin{proof}
    Let $n \coloneqq \dim(X)$, and $m \coloneqq n - s =  \dim(Y)$. Furthermore, let $U \coloneqq X \setminus D$ and $V \coloneqq Y \setminus \Delta$ (note that $f^{-1}(V) = U$ by assumption). We will denote by $i$ the inclusion of $U$ into $X$, by $j$ the inclusion of $V$ into $Y$, and by $g \colon U \to V$ the morphism induced by $f$.
    
    By the decomposition theorem for pure Hodge modules (which follows from \cite[Théorème 1]{Saito_Modules_de_Hodge_polarisables} by \cite[Théorème 1.5]{Deligne_Theoreme_de_Lefschetz_et_criteres_de_degenerescence_de_suite_spectrales}), we have an isomorphism in the derived category of filtered $\cD_V$--modules \[ g_+\omega_U \cong \bigoplus_{i \in \bZ}  \cH^ig_+\omega_U[-i] \] (each $\cH^ig_+\omega_U$ acquires a filtration by strictness, see \cite[Théorème 1]{Saito_Modules_de_Hodge_polarisables}). Applying $j_+$ and $\Gr^F_{-n}$ to this decomposition gives an isomorphism \[ \Gr^F_{-n}(f_+i_+\omega_U) \cong \bigoplus_{i \in \bZ} \Gr^F_{-n}(j_+\cH^i(g_+\omega_U))[-i]. \] 
    Since all complexes $\cM^{\bullet}$ above satisfy $\Gr_i^F(\cM^{\bullet})$ = 0 for all $i < -n$ by construction and \autoref{preparation_lemma_HM}.\autoref{itm:filtration_on_pushforward}, we have that $\Gr_{-n}^F(j_+\cH^i(g_+\omega_U)) = F_{-n}(j_+\cH^i(g_+\omega_U))$ is concentrated in degree zero by \autoref{preparation_lemma_HM}.\autoref{itm:exactness_open_immersion} for all $i \in \bZ$, and that $\Gr^F_{-n}(f_+i_+\omega_U) = \mathrm{Gr}_{-n}^F(\mathrm{DR}_Y(f_+i_+\omega_U))$ (see e.g. \cite[Page 5]{Popa_Kodaira_Saito_vanishing_and_applications} for the definition of the induced filtration on the de Rham complex). The strictness property in \cite[Theorem 2.14]{Saito_Mixed_Hodge_modules} therefore gives an isomorphism \[ \mathrm{Gr}_{-n}^F(\mathrm{DR}_Y(f_+i_+\omega_U)) \cong Rf_*\mathrm{Gr}_{-n}^F(\mathrm{DR}_X(i_+\omega_U)) = Rf_*(F_{-n}(i_+\omega_U)). \] Given that $F_{-n}(i_+\omega_U) \cong \omega_X(D)$ by \autoref{preparation_lemma_HM}.\autoref{itm:case_of_trivial_HM}, we obtained the formula \[ Rf_*\omega_X(D) \cong \bigoplus_{i \in \bZ} \: F_{-n}(j_+\cH^i(g_+\omega_U))[-i]. \] In particular, we have proven that $Rf_*\omega_X(D) \cong \bigoplus_{i \in \bZ} R^if_*\omega_X(D)[-i]$. Let us move to the torsion--freeness result: by the above, we have to show that each $F_{-n}(j_+\cH^i(g_+\omega_U))$ is torsion--free. All components $\cM$ of $\cH^ig_+\omega_U$ with strict support not equal to $V$ satisfy that $p(\cM) > p(\cH^ig_+\omega_U)$ by \cite[Proposition 2.6]{Saito_On_Kollar_conjecture}, so in particular $F_{-n}(\cM) = 0$. If $\cN$ denotes the component of $\cH^i(g_+\omega_U)$ with strict support on $V$, we then deduce from \autoref{preparation_lemma_HM}.\autoref{itm:filtration_on_pushforward} that $F_{-n}(j_+\cH^i(g_+\omega_U)) \cong F_{-n}(j_+\cN)$. Since $\cN$ underlies a pure Hodge module, we can write it as the minimal extension of some vector bundle $\cV$ with a connection on some open $V'$ with associated open immersion $j' \colon V' \hookrightarrow V$ (see \cite[Theorem 3.21]{Saito_Mixed_Hodge_modules}), so in particular $\cN \inc j'_*\cV$. Hence, $F_{-n}(j_+\cN) \inc j_+\cN = j_*\cN \inc j_*j'_*\cV$, so it is torsion--free.
    
    Assume that $Y$ is smooth, that $\Delta$ is snc and that $f_*\cO_X = \cO_Y$. We are left to show that $R^sf_*\omega_X(D) \cong \omega_Y(\Delta)$, or equivalently that $F_{-n}(j_+\cH^s(g_+\omega_U)) \cong \omega_Y(\Delta)$. If we let $h \colon \ttilde{U} \to \ttilde{V}$ denote the restriction of $g$ to its smooth locus, then it is straightforward to see that $\cH^{-s}(h_+\omega_{\ttilde{U}}) \cong \omega_{\ttilde{V}}$ as filtered $\cD_{V'}$--modules (we are using the assumption that $f$ has connected fibers). Since $\cH^{-s}g_+\omega_U$ underlies a pure Hodge module, we obtain by \cite[Theorem 3.21]{Saito_Mixed_Hodge_modules} that its component of strict support on $V$ is precisely $\omega_V$. Given that there is an isomorphism $\cH^{-s}g_+\omega_U \to (\cH^sg_+\omega_U)(s)$ by \cite[Théorème 1]{Saito_Modules_de_Hodge_polarisables}, we obtain that the component of $\cH^sg_+\omega_U$ of strict support on $V$ is precisely $\omega_V(-s)$. We can therefore write $\cH^s(g_+\omega_U) \cong \omega_V(-s) \oplus \cM'$, where $\cM'$ is not supported everywhere and satisfies $p(\cM') > p(\omega_V(-s)) = -n$ (recall our previous discussion), so combining points \autoref{itm:filtration_on_pushforward} and \autoref{itm:case_of_trivial_HM} of \autoref{preparation_lemma_HM} gives \[ F_{-n}(j_+\cH^s(g_+\omega_U)) \cong F_{-n}(j_+\omega_V(-s)) = F_{-m}(j_+\omega_V) = \omega_Y(\Delta). \qedhere \]
\end{proof}

\begin{remark}
    After discussing with Fujino, we realized that the torsion-freeness can be deduced from applying \cite[Theorem 6.3.(i)]{Fujino_Fundamental_theorems_for_the_log_MMP} to $B = D - \varepsilon f^*\Delta$ with $0 < \varepsilon \ll 1$ and $L = K_X + D$. In the case where $(Y, \Delta)$ is log smooth, it can be proven in arbitrary characteristic that there is a natural surjective map $R^jf_*\omega_X(D) \to \omega_Y(\Delta)$ (for example reduce to positive characteristic, and use methods from test ideals/Cartier modules), so we can deduce from the previously obtained torsion-freeness that this map is an isomorphism if $f_*\cO_X = \cO_Y$.
\end{remark}
\section{Varieties of maximal Albanese dimension}\label{sec:maxAlb}

In this section we will prove \autoref{main_thm_max_alb_dim}. To this aim, fix a smooth variety $V$ of maximal Albanese dimension with good compactification $(X, D)$, and let us assume that $\overline{P}_2(V) = 1$. As in \autoref{semi-alb}, let $A(V)$ be the Albanese variety of $V$ and $(Z,\Delta)$ its standard compactification. We set $n \coloneqq \dim(V)$ and $r \coloneqq \overline{q}(V) - q(X) \geq 0$.

The main part of the proof will be to show that $\alb_V \colon V \to \Alb(V)$ is birational, and we will show this fact by induction on $r$ (the WWPB property will come afterwards). Note that when $r = 0$, then $X$ has in fact maximal Albanese dimension too, so the result follows from \cite{Chen_Hacon_Characterization_of_abelian_varieties} (see also \cite[Theorem 4.1]{Pareschi_Basic_results_on_irr_vars_via_FM_methods}, \cite[Theorem 3.1]{Jiang_An_effective_version_of_a_thm_of_Kawamata_on_the_Albanese_map} or \cite[Theorem C]{Baudin_Effective_Characterization_of_ordinary_abelian_varieties}). Thus, we will assume from now on that $r > 0$ in order to prove birationality.

Let us write $\ker(\Alb(V) \to \Alb(X)) = \bG_{m, 1} \times \dots \times \bG_{m, r}$. Given $1 \leq i \leq r$, we can fix a $\bP^1$--fibration $(Z, \Delta) \to (T_i, \Delta_{T_i})$ over $\Alb(X)$ such that $G_i \coloneqq T_i \setminus \Delta_{T_i} = \Alb(V)/\bG_{m, i}$. In order to ease notation, and because it should not create confusion, we will write $(T, \Delta_T): = (T_i, \Delta_{T_i})$. Up to blowing up $X$, we have the following commutative diagram:
\begin{equation}\label{eq:setupdiag}
    \begin{tikzcd}
        V \arrow[d, "\alb_V"] \arrow[rr, hook]     &  & X \arrow[d, "g"'] \arrow[ddd, "a", bend left=49] \arrow[dd, "f", bend left] \\
        \Alb(V) \arrow[rr, hook] \arrow[d, two heads] &  & Z \arrow[d, "\pi"'] \\
        \Alb(V)/\bG_{m, 1} \arrow[rr, hook]                &  & T \arrow[d, "\theta"'] \\
        &  & \Alb(X).   
    \end{tikzcd}
\end{equation}

Our approach to show birationality will be to show that $g \colon X \to Z$ is étale over $\Alb(V)$ (up to $g$--exceptional components). This will conclude the proof, since a finite étale cover of a semi-abelian variety is again semi-abelian. We refer the reader to the proof of \autoref{birational_mad_case} for the precise details. 

\begin{lemma}\label{dominance_in_mad_case}
    The morphism $g \colon X \to Z$ is surjective (equivalently, $\alb_V$ is dominant).
\end{lemma}
\begin{proof}
    Since $g \colon (X, D) \to (Z, \Delta)$ is generically finite (onto its image), the pullback morphism $g^*\Omega^1_Z(\log \Delta) \to \Omega^1_X(\log D)$ is generically surjective. By triviality of $\Omega^1_Z(\log \Delta)$, this shows that $\Omega^1_X(\log D)$ is generically generated by its global sections. Given that $h^0(X, \Omega^1_X(\log D)) = \dim(Z)$ by construction (\cite[Section 3]{Fujino_On_quasi_Albanese_maps}), we deduce by \cite[Lemma 4.2]{Zhang_Abundance_for_non_uniruled_3_folds_with_non_trivial_albanese_in_pos_char} and the assumption $\overline{P}_2(V) = 1$ that $\dim(Z) = \rk(\Omega^1_X(\log D)) = \dim(X)$.
\end{proof}
\begin{remark}\label{firstP}
    Observe that the maximal Albanese dimension assumption also ensures that  $\overline{P}_1(V) = 1$, since pulling back log differential forms and using \autoref{dominance_in_mad_case} gives an inclusion $\cO_Z \cong \omega_Z(\Delta) \hookrightarrow g_*\omega_X(D)$.
\end{remark}

The following lemma is the only one using the induction hypothesis.

\begin{lemma}\label{alb_of_X_fibration}
    We have that $f_*\cO_X = \cO_T$.
\end{lemma}
\begin{proof} 
    Let \[ \begin{tikzcd}
    X \arrow[rr, "\mu"] \arrow[rrrr, "f", bend left] &  & S \arrow[rr, "h"] &  & T.
    \end{tikzcd} \] be a smooth model of the Stein factorization of $f$, namely $S$ is smooth, $D_S \coloneqq h^{-1}(\Delta_T)$ is simple normal crossing and $h$ is generically finite and surjective (we are using \autoref{dominance_in_mad_case}). We aim to prove that $h$ is birational.
    
    Let $W$ be the complement of $D_S$ in $S$. Then $\dim W=n-1$ and $h_{|W}\colon W\rightarrow G$ is a dominant generically finite map to a semi--abelian variety of dimension $n-1$ (hence $W$ has maximal Albanese dimension). We are going to show that $\overline{q}(W) - q(S) = r - 1$, that $\overline{P}_2(W)=1$ and that $h_{|W}$ is the Albanese map of $W$. The induction hypothesis will then ensure that $h_{|W}$ birational, concluding the proof. 
    
    By the universal property of the Albanese morphism, we have a commutative diagram 
    \[ \begin{tikzcd}
        V \arrow[d] \arrow[rr, "\alb_V"]                        &                   & \Alb(V) \arrow[d, two heads, "\pi"] \arrow[ld, two heads] \\
        W \arrow[r, "\alb_W"] \arrow[rr, "h_{|W}"', bend right] & \Alb(W) \arrow[r] & G      
    \end{tikzcd} \] (note that $\Alb(V) \to  \Alb(W)$ is surjective, since $V \to W$ is dominant and $W$ generates its Albanese variety). In particular, we have that $\overline{q}(W) \in \{n, n - 1\}$. In the former case, we would have that the map $\Alb(V)\rightarrow \Alb(W)$ is generically finite, and hence so would be $V \to \Alb(W)$. This is impossible, since $V \to W$ is not, so we have proven that $\overline{q}(W)=n-1=\dim W$. In particular, the morphism $\Alb(W) \to G$ is finite (recall that $W$ dominates $G$). Since $\Alb(V) \to G$ has connected fibers and factors through $\Alb(W)$, we deduce that $\Alb(W) \to G$ is an isomorphism. This shows in particular that $\overline{q}(W)-q(S)=n-1-q(X)=r-1$, so we are left to show that $\overline{P}_2(W) = 1$.
    
    Choose a nowhere vanishing logarithmic $(n - 1)$--form $\psi_T \in H^0(T, \omega_T(\Delta_T))$ and pick $\eta \in H^0(Z, \Omega^1_Z(\log \Delta))$ such that $\pi^*(\psi_T) \wedge \eta$ is nowhere vanishing. We define a morphism $$\alpha\colon \omega_S(D_S)\rightarrow \mu_*\omega_X(D)$$ via $\psi \mapsto \mu^*\psi\wedge\eta$, where $\psi$ is a local logarithmic $(n - 1)$--form. Note that $\alpha$ is non--trivial since the image of $h^*(\psi_T)$ is non-zero, so it is automatically injective. In particular, $\overline{P}_2(W) \leq \overline{P}_2(D) = 1$, so the proof is complete.
\end{proof}

%%%%%%%%%%%%%%%%%%%%%%%%%%%%%%%%%%%%%%%%%%%%%%%%%%%%%%%%%%%%%%

\begin{notation}
   From now on, we let $H \coloneqq g^*(\Delta)_{\red} \inc D$, $H_T \coloneqq f^*(\Delta_T)_{\red} \inc H$ and $H' \coloneqq H - H_T$. We also set $\Delta'\coloneqq\Delta-\pi^*\Delta_T = \Delta_1' + \Delta_2'$.
\end{notation} Since $g$ is generically finite, we have the logarithmic ramification formula  (see \cite[\S 13]{iitaka1977logarithmic}):
\begin{equation}\label{logramification}
    K_X + D \sim g^*(K_Z + \Delta) + \overline{R_g} \sim \overline{R_g}. 
\end{equation} 
Let us also denote by $R_g$ the usual ramification divisor of $g$. The following results will be useful:

\begin{lemma}\label{basic_things}
    The following statements hold:
    \begin{enumerate}
        \item\label{only_one_section}  $h^0(X, \omega_X(D + \overline{R_g})) = 1$;
        \item\label{the_section_has_poles_on_the_wholeH} $h^0(X, \omega_X(H)) = 1$, and the non--zero global section has poles at each generic point of $H$ belonging to components not contracted by $g$; 

        \item\label{MLPT_lemma} for any irreducible divisor $\Gamma \not\subseteq H$, we have that $\Gamma \inc \overline{R_g}$ if and only if $\Gamma \inc D + R_g$;
        \item\label{exceptional} any $g$--exceptional prime divisor is in $H+\overline{R_g}$.    \end{enumerate}
\end{lemma}
\begin{proof}
    It follows from \autoref{logramification} that 
    $$K_X+D+\overline{R_g}\sim 2(K_X+D),$$
    so \autoref{only_one_section} follows from the assumption that $\overline{P}_2(V)=1$. Now we prove \autoref{the_section_has_poles_on_the_wholeH} by mimicking the argument in the proof of \cite[Proposition 5.2]{MLPT2023}. The space $H^0(Z,\Omega_Z^1(\log \Delta))$ is generated by logarithmic forms $\eta_1,\ldots,\eta_n$ such that $\omega \coloneqq \eta_1\wedge\cdots\wedge\eta_n$ vanishes nowhere on $\Alb(V)$ and has poles on the whole $\Delta$. Since $g$ is surjective we have that $g^*\omega$ is a nonzero global logarithmic form on $X$ whose poles are contained in $H$ by construction. In particular we get
    $$0\neq h^0(X,\omega_X(H)))\leq h^0(X,\omega_X(D))=1,$$
    where the last equality is provided by \autoref{firstP}.
    Now, if $\Gamma$ is a component of $H$ not contracted by $g$, then a local computation shows that $g^*\omega$ has a pole along $\Gamma$, so \autoref{the_section_has_poles_on_the_wholeH} is proven. The proof of \autoref{MLPT_lemma} is a local computation, see \cite[Lemma 2.8]{MLPT2023}. For \autoref{exceptional}, let $E$ be a $g$--exceptional divisor that is not contained in $H$. Then we have that $g(E)\not\subseteq \Delta$. By \autoref{MLPT_lemma}, the statement is proven if we show that $E\subseteq R_g$. Let $p \in g(E)$ be a point, and set $E_p \coloneqq g^{-1}(p)$. We need to show that $E_p \inc R_g$. Consider the exact sequence of differential forms
    $$g^*\Omega^1_Z\rightarrow\Omega^1_X\rightarrow \Omega^1_{X/Z}\rightarrow 0.$$
    Given that $\Omega^1_{X/Z}|_{E_p} \cong \Omega^1_{E_p/p}$ has positive rank, the map $g^*\Omega^1_Z\rightarrow\Omega^1_X$ is not surjective at $E_p$. Taking determinants gives that $g^*\omega_Z\rightarrow \omega_X$ is not surjective at any point of $E_p$, i.e. $E_p \inc R_g$.
\end{proof}

Recall that our objective is to show that $g$ is étale over $\Alb(V)$ (up to $g$--exceptional components). In other words, we need to show $\overline{R_g}$ is as small as possible. Here is the first step towards this goal. 

\begin{proposition}\label{no_bad_component_dominates}
    Let $\Gamma$ be an irreducible component of $\overline{R_g}$ such that $\Gamma \not\inc H$. Then $\Gamma$ does not dominate $T$.
\end{proposition}

\begin{proof}
    We will proceed by contradiction, so assume that there exists some component $\Gamma \inc \overline{R_g}$ not in $H$ dominating $T$. Consider the exact sequence \[ \begin{tikzcd}
        0 \arrow[r] & \omega_X(H_T) \arrow[r]  & \omega_X(H + \Gamma) \arrow[r]  & \omega_{H' + \Gamma}(H_T|_{H' + \Gamma}) \arrow[r] &  0.
    \end{tikzcd} \]
    Pushing by $f$ and using \autoref{hodge_module_lemma} (together with \autoref{alb_of_X_fibration}) gives an exact sequence 
    \[ \begin{tikzcd}
        0 \arrow[r] &   f_*\omega_X(H_T) \arrow[r]   & f_*\omega_X(H + \Gamma) \arrow[r]  & f_*\omega_{H' + \Gamma}(H_T|_{H' + \Gamma}) \arrow[r]  & \omega_T(\Delta_T), 
    \end{tikzcd} \]
    which we split into two exact sequences 
    \begin{equation}\label{firstequationnondominance} \begin{tikzcd}
        0 \arrow[r] & f_*\omega_X(H_T) \arrow[r] & f_*\omega_X(H + \Gamma) \arrow[r] & \cM \arrow[r] & 0,
    \end{tikzcd} \end{equation} and 
    \begin{equation}\label{secondequationnondominance} \begin{tikzcd}
        0 \arrow[r] & \cM \arrow[r] & f_*\omega_{H' + \Gamma}(H_T|_{H' + \Gamma}) \arrow[r] & \cN \arrow[r] & 0
    \end{tikzcd} \end{equation} with $\cN \inc \omega_T(\Delta_T) \cong \cO_T$. Let $H_1'$ and $H_2'$ be components of $H'$ dominating $\Delta_1'$ and $\Delta_2'$ respectively, and let $\Lambda = H' - H_1' - H_2'$. Note that applying Grothendieck duality to the inclusion \[ \cO_{H' + \Gamma}(H_T|_{H' + \Gamma}) \inc \cO_{H_1'}(H_T|_{H_1'}) \oplus \cO_{H_2'}(H_T|_{H_2'}) \oplus \cO_{\Gamma}(H_T|_{\Gamma}) \oplus \cO_{\Lambda}(H_T|_{\Lambda}) \] and \stacksproj{0A7U} to its cokernel gives an inclusion
    \[ \omega_{H_1'}(H_T|_{H_1'}) \oplus \omega_{H_2'}(H_T|_{H_2'}) \oplus \omega_{\Gamma}(H_T|_{\Gamma}) \oplus \omega_{\Lambda}(H_T|_{\Lambda}) \inc \omega_{H' + \Gamma}(H_T|_{H' + \Gamma}). \] Since each map $(H_i', H_T|_{H_i'}) \to (T, \Delta_T)$ and $(\Gamma, H_T|_{\Gamma}) \to (T, \Delta_T)$ are generically finite by assumption, we deduce by the same pullback argument as in \autoref{firstP} that $h^0(T, f_*\omega_{H' + \Gamma}(H_T)) \geq 3$. In particular, \autoref{secondequationnondominance} gives $h^0(T, \cM) \geq 2$, since $\cN \inc \cO_T$.

   If we knew that $H^1(T, f_*\omega_X(H_T)) = 0$, then we would deduce that $h^0(X, \omega_X(H + T)) \geq 2$ by \autoref{firstequationnondominance}, contradicting \autoref{basic_things}.\autoref{only_one_section} and thereby concluding the proof. Let us therefore show this vanishing. Given that $H_T = f^*(\Delta_T)_{\red}$, there exists $0 < \varepsilon \ll 1$, such that $B \coloneqq H_T - \varepsilon f^*\Delta_T$ is effective. Since $\Delta_T$ is $\theta$--ample, we obtain by \cite[Theorem 6.3.(ii)]{Fujino_Fundamental_theorems_for_the_log_MMP} that \[ R^j\theta_*(f_*\omega_X(H_T)) = 0 \] for all $j > 0$, so in particular $H^1(T, f_*\omega_X(H_T)) \cong H^1(\Alb(X), a_*\omega_X(H_T))$. Given that $H^0(\Alb(X), a_*\omega_X(H_T)) = 0$ by \autoref{basic_things}.\autoref{the_section_has_poles_on_the_wholeH} and $a_*\omega_X(H_T)$ is a GV--sheaf, we obtain that also $H^1(\Alb(X), a_*\omega_X(H_T)) = 0$.
\end{proof}

\begin{corollary}\label{log-erale}
    The restriction of $g\colon (X, D) \to (Z, \Delta)$ to a general fiber $F$ of $X \to T$ is log étale. In particular, $(F, D|_F) \cong (\bP^1, 0 + \infty)$.
\end{corollary}
\begin{proof}
    By the logarithmic ramification formula and \autoref{no_bad_component_dominates}, we obtain that the restriction $(F, D|_F) \to (\bP^1, \Delta|_{\bP^1})$ of $g$ to a general fiber of $f$ is étale over $\bG_m = \bP^1 \setminus \Delta|_{\bP^1}$. However, any finite étale cover of $\bG_m$ is isomorphic to $\bG_m$ (the map being given by taking $m$'th power for some $m > 0$), so we automatically deduce that $(F, D|_F) \cong (\bP^1, 0 + \infty)$ (one could also have done a Riemann--Hurwitz computation and obtain the same conclusion).
\end{proof}

\iffalse
\begin{proof}
    \JB{You had written a proof, but I don't think we need to prove this fact (I commented your proof though). What do you think?}
    Let $\Lambda$ be the union of all the ireducible components of $\overline{R_g}$ not contained in $H$ \JB{We should put a $D$ and not an $H$ I think!!!! Going to the actual $H$ will be the purpose of the WWPB part of the argument!!}. Then, by \autoref{no_bad_component_dominates},
    we have that $f(\Lambda)\neq T$. Let us take $ t\in T\backslash f(\Lambda)$ general such that $F_t$, the fiber of $F$ over $t$, is smooth. Then $$(K_X+D+\overline{R_g})\cdot F_t=(K_X+H)\cdot F_t =g^*(K_Z+\Delta)\cdot F_t.$$ We deduce that the restriction of $g$ to $F_t$ is log-\'etale, since its log-ramification divisor is trivial.
\end{proof}

\fi

\begin{notation}
    Let \[ \begin{tikzcd}
    X \arrow[r, "\mu"] & \overline{X} \arrow[r, "\overline{g}"] & Z
    \end{tikzcd} \] denote the Stein factorization of $g$, and let $\overline{f} \coloneqq \pi \circ \overline{h} \colon \overline{X} \to T$. We denote by $\Lambda$ the union of all components of the usual ramification divisor $R_g$ of $g$ that are not contained in $H$ and that are not $g$--exceptional. We then set $U \coloneqq T \setminus (f(\Lambda) \cup \Delta_T) \inc T \setminus \Delta_T = G$. Finally, we define $Z_U \coloneqq \pi^{-1}(U) \cap \Alb(V) \inc \Alb(V)$, and $\overline{X}_U \coloneqq \overline{g}^{-1}(Z_U) \inc \overline{X}$.
\end{notation}

Thanks to \autoref{no_bad_component_dominates}, $U \neq \emptyset$. By construction, we have a finite étale morphism $\overline{X}_U \to Z_U$. Our strategy to conclude the proof of birationality will be to extend it to a finite étale cover of $\Alb(V)$. We will achieve this by studying precisely the finite index subgroup of $\pi_1(Z_U)$ corresponding to this cover, and ultimately show that it is the preimage of a finite index subgroup of $\pi_1(\Alb(V))$ under $\pi_1(Z_U) \to \pi_1(\Alb(V))$.

\begin{lemma}\label{Galois_extension_and_nice_ramification}
   The finite étale morphism $\overline{X}_U \to Z_U$ is Galois. In particular, $g^{-1}(g(R_g)) \inc R_g$.
\end{lemma}
\begin{proof}
    By construction, the morphism $\overline{X}_U \to Z_U$ is finite étale, so it corresponds to a subgroup $N \inc \pi_1(Z_U)$ (namely the image of $\pi_1(\overline{X}_U) \to \pi_1(Z_U)$). To obtain the first statement, we then need to show that this subgroup is normal by the Galois correspondence for covering spaces.

    Since both $\overline{X}_U \to U$ and $Z_U \to U$ are smooth, they are fibrations in the sense of topology. Let $F^{\circ}$ denote a fiber of $\overline{X}_U \to U$. We therefore have a morphism of exact sequences of homotopy groups 
	\[ \begin{tikzcd}
		& \pi_1(F^{\circ}) \arrow[r] \arrow[d] & \pi_1(\overline{X}_U) \arrow[r] \arrow[d] & \pi_1(U) \arrow[r] \arrow[d, "\cong"] & 1 \\
		1 \arrow[r] & \pi_1(\bG_m) \arrow[r]       & \pi_1(Z_U) \arrow[r]                      & \pi_1(U) \arrow[r]                    & 1
	\end{tikzcd} \]
	(the map $\pi_1(\bG_m) \to \pi_1(Z_U)$ is injective, since $Z_U$ is homeomorphic to $U \times \bG_m$ over $U$ by \autoref{sequence_Lie_groups_split}). A diagram chasing argument shows that the left cosets of $N \inc \pi_1(Z_U)$ are of the form $gN$ with $g \in \pi_1(\bG_m) \inc \pi_1(Z_U)$. Since this is a central subgroup, this shows that left and right cosets of $N$ coincide, i.e. that $N$ is a normal subgroup. The snake lemma then shows that the cokernel of $N \inc \pi_1(Z_U)$ is cyclic.

    Now, let $K = \langle \gamma \rangle$ denote the Galois group of this morphism, and let $R_{\overline{g}}$ denote the ramification divisor of $\overline{g}$. We can extend the action of $\gamma$ on an open subset of $\overline{X}$ with complement of codimension at least $2$, showing that $\overline{g}^{-1}(\overline{g}(R_{\overline{g}})) \inc R_{\overline{g}}$ since $K$ acts transitively on preimages of codimension one points of $Z$ (quotients by finite groups are good quotients). Since all $g$--exceptional divisors of $X$ are contained in $R_g$ by the proof of \autoref{basic_things}.\autoref{exceptional}, we deduce that $g^{-1}(g(R_g)) \inc R_g$.
\end{proof}

Recall that an effective divisor $\Delta_B$ on an abelian variety $B$ is called a \emph{principal polarization} if it is ample and satisfies $h^0(B, \cO_B(\Delta_B)) = 1$. 

\begin{lemma}\label{principal_polarization}
    Assume that $\Lambda \neq 0$. Then there exists a fibration $\phi \colon Z \to B$ to an abelian variety and a principal polarization $\Delta_B \inc B$ such that $g(\Lambda) = \phi^{-1}(\Delta_B)$.
\end{lemma}
\begin{proof}
    First of all, let us show by contradiction that $\Lambda$ does not dominate $A(X)$. Let $K$ denote the field of fractions of $A(X)$. If $\Lambda$ dominated $A(X)$, we would obtain that the generic fiber $g(\Lambda)_K$ of $g(\Lambda) \to A(X)$ is a non-zero divisor in $(\bP^1_K)^r$. But then, it must dominate one of the standard projections $(\bP^1_K)^r \to (\bP^1_K)^{r - 1}$, contradicting \autoref{no_bad_component_dominates} (recall from the notations at the very beginning of this section that $T$ can be picked as the compactification of \emph{any} or the $\bG_m$--quotients from $\Alb(V)$).

    Thus, we obtained that $\Lambda$ does not dominate $A(X)$, so $a(\Lambda) \inc A(X)$ must be a divisor for dimensional reasons. We then obtain the existence of a fibration $\psi \colon A(X) \surj B$ of abelian varieties such that $a(\Lambda) = \psi^*\Delta_B$ with $\Delta_B \inc B$ ample (we really mean an equality of divisors, not a linear equivalence). Set $b \coloneqq \psi \circ a \colon X \to B$, and $\phi \coloneqq \psi \circ \theta \circ \pi \colon Z \to B$. Given that all fibers are $\phi$ are irreducible, we have that $\phi^*(\phi(g(\Lambda))) = g(\Lambda)$, or in other words that $\phi^{-1}(\Delta_B) = g(\Lambda)$. Hence, we are therefore left to show that $h^0(B, \cO_B(\Delta_B)) = 1$. To see this, we will construct an injection $\cO_B(\Delta_B) \hookrightarrow b_*\omega_X(H \cup \overline{R_g})$, where $H \cup \overline{R_g} \coloneqq (H + \overline{R_g})_{\red}$.

    Let us first show that $b_*\omega_X(H \cup \overline{R_g})$ is a line bundle, which will be useful to construct this morphism later on. Note that this sheaf has rank $1$ by \autoref{log-erale}, so it is enough to show that it is locally free. By \autoref{Galois_extension_and_nice_ramification}, we have that $b^{-1}(b(\Lambda)) = g^{-1}(g(\Lambda)) \inc R_g$, so $b^{-1}(\Delta_B) \inc \Lambda \cup \mathrm{Exc}(g)$. Since $\Lambda \inc \overline{R_g}$ by \autoref{basic_things}.\autoref{MLPT_lemma}, we deduce by \autoref{basic_things}.\autoref{exceptional} that $b^{-1}(\Delta_B) \inc \overline{R_g}$. Using that $\Delta_B$ is ample, the same trick via \cite[Theorem 6.3.(ii)]{Fujino_Fundamental_theorems_for_the_log_MMP} as in the end of the proof of \autoref{no_bad_component_dominates} yields that for all $i > 0$ and $\alpha \in \Pic^0(B)$, \[ H^i(B, b_*\omega_X(H \cup \overline{R_g}) \otimes \alpha) = 0. \] In particular, $\FM_B(b_*\omega_X(H \cup \overline{R_g}))$ is locally free on $\bighat{B}$ by \autoref{FM_and_cohom}. Since $h^0(B, b_*\omega_X(H \cup \overline{R_g})) = 1$ by \autoref{basic_things}.\autoref{only_one_section}, we obtain that $\FM_B(b_*\omega_X(H \cup \overline{R_g}))$ is in fact a line bundle by \autoref{FM_and_cohom}. Write $\cL \coloneqq \FM_B(b_*\omega_X(H \cup \overline{R_g}))$. Since $b_*\omega_X(H \cup \overline{R_g}) \cong \FM_{\bighat{B}}(\cL)$ is supported in degree zero and is torsion-free, we deduce by \cite[Proposition 2.8]{Pareschi_Popa_Regularity_on_abelian_varieties_III} and \cite[Example 3.10]{Pareschi_Popa_M_regularity_and_the_Fourier_Mukai_transform} that $\cL$ is ample. By Kodaira vanishing and \autoref{FM_and_cohom}, we deduce that $\Tor_i(b_*\omega_X(H \cup \overline{R_g}), \bC(x)) = 0$ for all $x \in A$ and $i > 0$. In other words, we have proven that $b_*\omega_X(H \cup \overline{R_g})$ is itself locally free. 

    Let us now construct the injection $\cO_B(\Delta_B) \hookrightarrow b_*\omega_X(H \cup \overline{R_g})$. Given that both sheaves are line bundles, it is enough to verify that the cokernel of $\cO_B \hookrightarrow b_*\omega_X(H \cup \overline{R_g})$ (coming from the unique non-zero section of this sheaf) contains each codimension one point of $\Delta_B$. Let us fix an irreducible component $\Gamma$ of $\Delta_B$, and note that $H + b^{-1}_*\Gamma \leq H \cup \overline{R_g}$, since $b^{-1}_*\Gamma \not\subseteq H$ by construction. We will show that the generic point of $\Gamma$ lies in the cokernel of the unique (non-zero) map $\cO_B \to b_*\omega_X(H + b^{-1}_*\Gamma)$ (this latter sheaf embeds in $b_*\omega_X(H \cup \overline{R_g})$ by the previous sentence).
    
    Since $R^ig_*\omega_X(H) = 0$ for all $i > 0$ by \autoref{hodge_module_lemma}, we have an exact sequence 
    \[ \begin{tikzcd}
        0 \arrow[r] & g_*\omega_X(H) \arrow[r]  & g_*\omega_X(H + b^{-1}_*\Gamma) \arrow[r] & g_*\omega_{b^{-1}_*\Gamma}(H|_{b^{-1}_*\Gamma}) \arrow[r] & 0
    \end{tikzcd} \] on $Z$. Recall that $\Delta$ is $(\theta \circ \pi)$-ample, so the same argument using \cite[Theorem 6.3.(ii)]{Fujino_Fundamental_theorems_for_the_log_MMP} as in the end of the proof of \autoref{no_bad_component_dominates} shows that $R^i(\theta \circ \pi)_*(g_*\omega_X(H)) = 0$ for $i > 0$, so we have an exact sequence 
    \[ \begin{tikzcd}
        0 \arrow[r] & a_*\omega_X(H) \arrow[r]  & a_*\omega_X(H + b^{-1}_*\Gamma) \arrow[r]  & a_*\omega_{b^{-1}_*\Gamma}(H|_{b^{-1}_*\Gamma}) \arrow[r]  & 0
    \end{tikzcd} \] on $A(X)$. By \autoref{lem:GVmorphism}.\autoref{lem:GVmorphismitem}, the sheaf $a_*\omega_X(H)$ contains a unipotent vector bundle as a direct summand. However, this sheaf has rank one by \autoref{log-erale}, so we deduce that $a_*\omega_X(H) \cong \cO_A$ by torsion--freeness. Pushing forward to $B$ then gives an exact sequence \[ \begin{tikzcd}
        0 \arrow[rr] &  & \cO_B \arrow[rr]          &  & b_*\omega_X(H + b^{-1}_*\Gamma) \arrow[rr] &  & b_*\omega_{b^{-1}_*\Gamma}(H|_{b^{-1}_*\Gamma}) \connecting{dllll}{""} \\
        &  & R^1\psi_*\cO_A \arrow[rr] &  & \dots &  &        
    \end{tikzcd} \] Given that $R^1\psi_*\cO_A$ is torsion--free while $b_*(\omega_{b^{-1}_*\Gamma}(H|_{b^{-1}_*\Gamma}))$ is torsion on $B$, we deduce that the connecting morphism above is zero. However, the morphism $(b^{-1}_*\Gamma, H|_{b^{-1}_*\Gamma}) \to (Z, \Delta)$ is generically finite, whence $0 \neq H^0(b^{-1}_*\Gamma, \omega_{b^{-1}_*\Gamma}(H|_{b^{-1}_*\Gamma})) = H^0(B, b_*(\omega_{b^{-1}_*\Gamma}(H|_{b^{-1}_*\Gamma})))$. Thus, the sheaf $b_*(\omega_{b^{-1}_*\Gamma}(H|_{b^{-1}_*\Gamma}))$ is in particular non--zero, so it must be supported everywhere on $\Gamma$. This shows that the cokernel of $\cO_B \hookrightarrow b_*\omega_X(H + b^{-1}_*\Gamma)$ contains the generic point of $\Gamma$, so the proof is complete.
\end{proof}

\begin{proposition}\label{general_result_first_homology_groups}
    Let $B$ be an abelian variety, and let $\Delta_B$ be a principal polarization on $B$. Then the morphism of homology groups $H_1(B \setminus \Delta_B, \bZ) \to H_1(B, \bZ)$ is an isomorphism.
\end{proposition}
\begin{proof}
    By \cite[Theorem 4.3.1]{Complex_abelian_varieties} and the Künneth formula, we may assume that $\Delta_B$ is irreducible. Write $W \coloneqq B \setminus \Delta_B$. Since $\pi_1(W) \to \pi_1(B)$ is surjective ($B$ is normal, so the restriction of a connected covering space of $B$ to $W$ is still connected), the map on abelianizations $H_1(W, \bZ) \to H_1(B, \bZ)$ is also surjective. Assume by contradiction that this morphism is not injective, and let $M \inc H_1(W, \bZ)$ be a subgroup surjecting onto $H_1(B, \bZ)$ such that $H_1(W, \bZ)/M$ is cyclic of order $m > 1$. Taking the preimage of $M$ under the abelianization map $\pi_1(W) \to H_1(W, \bZ)$ gives rise to a Galois covering space $\ttilde{W} \to W$ with Galois group $\bZ/m\bZ$. Since $\ttilde{W}$ is quasi--projective by Riemann's existence theorem (see e.g. \cite[Théorème 5.1]{SGA1}), taking the normalization of $B$ in $\bC(\ttilde{W})$ gives a finite morphism $\phi \colon \ttilde{B} \to B$, such that the induced field extension is Galois. Note that $\phi$ cannot be étale since $M$ surjects onto $H_1(B, \bZ) = \pi_1(B)$, so $\phi$ must ramify over $\Delta_B$. By \cite[Corollary 2.46]{Kollar_Singularities_of_the_minimal_model_program}, there exists $0 < \alpha < 1$ such that $m\alpha \in \bZ$ and $L \in \Pic(B)$ such that $L^{\otimes m} \cong \cO_B(m\alpha \Delta_B)$. Write $\alpha = \frac{s}{m}$ with $0 < s < m$, so that $L^{\otimes m} \cong \cO_B(s\Delta_B)$. Applying the Riemann--Roch formula on $B$ together with Kodaira vanishing gives that \[ m^{\dim(B)} \cdot \chi(B, L) = \chi(B, L^{\otimes m}) = \chi(B, \cO_B(s\Delta_B)) = s^{\dim(B)}\chi(B, \cO_B(\Delta_B)) = s^{\dim(B)}, \] leading to the contradiction that $m \leq s$.
\end{proof}

\begin{corollary}\label{topological_lemma}
    The inclusion $Z_U \inc \Alb(V)$ induces an isomorphism of homology groups $H_1(Z_U, \bZ) \cong H_1(\Alb(V), \bZ)$.
\end{corollary}

\begin{proof}
    If $\Lambda = 0$, then $Z_U = A(V)$ and the assertion is immediate. Let us therefore assume that $\Lambda \neq 0$. By \autoref{principal_polarization}, there exists a fibration $\phi \colon Z \to B$ to an abelian variety $B$ such that $g(\Lambda) = \phi^{-1}(\Delta_B)$ for some principal polarization $\Delta_B$ on $B$. Set $W \coloneqq B \setminus \Delta_B$. Consider the induced morphism $\ttilde{\phi} \colon A(V) \to B$, which therefore is a fibration of semi-abelian varieties satisfying $\ttilde{\phi}^{-1}(W) = Z_U$. By \autoref{sequence_Lie_groups_split}, there is an homeomorphism $A(V) \cong B \times F$ over $B$, where $F$ is a fiber of $\ttilde{\phi}$. This also induces a homeomorphism $Z_U \cong W \times F$, so it is enough to show that $H_1(W, \bZ) \to H_1(B, \bZ)$ is an isomorphism by the Künneth formula. This follows from \autoref{general_result_first_homology_groups}. \qedhere
    
    %(again, here we mean an e consider a fibration of abelian varieties $\psi \colon A(X) \surj B$ such that $R' = \psi^*(R'')$ for some ample effective divisor $R''$ on $B$, and set $W \coloneqq B \setminus R''$. We then deduce from \autoref{principal_polarization} that $R''$ is a principal polarization, so the inclusion $H_1(W, \bZ) \to H_1(B, \bZ)$ is an isomorphism by \autoref{general_result_first_homology_groups}. Consider the composition $h \coloneqq \psi \circ \theta \circ \pi \colon \Alb(V) \to B$ and note that by construction, $Z_U = h^{-1}(W)$. Let $H$ be a fiber of $h$. Then by \autoref{sequence_Lie_groups_split}, there exists a homeomorphism $\Alb(V) \cong B \times H$ over $B$. This isomorphism also induces an isomorphism $Z_U \cong B \times W$ over $W$, so combining the Künneth formula and \autoref{Galois_extension_and_nice_ramification} gives the result.
\end{proof}

We now have all the tools to finish the proof of birationality.

\begin{proposition}\label{birational_mad_case}
	The Albanese morphism $V \to A(V)$ is birational.
\end{proposition}
\begin{proof}
    Our objective will be to show that $\overline{X} \to Z$ is étale over $A(V)$. This would imply that $\overline{g}^{-1}(\Alb(V))$ is a semi--abelian variety by \cite{Fujino_On_quasi_Albanese_maps}, whence $\alb_V \colon V \to \Alb(V)$ would factor through a semi--abelian variety. The universal property of the Albanese morphism would therefore force $\overline{g}^{-1}(\Alb(V)) \to \Alb(V)$ to be an isomorphism, thereby concluding the proof.
    
    By construction, the morphism $\overline{X}_U \to Z_U$ is finite étale, so it corresponds to a subgroup $N \inc \pi_1(Z_U)$. By \autoref{Galois_extension_and_nice_ramification}, this subgroup is normal and its quotient is cyclic. In particular, $N$ is the preimage of a subgroup $\overline{N}$ of $\pi_1(Z_U)_{\ab}$ under the abelianization map. Note that \autoref{topological_lemma} gives us that $\pi_1(Z_U)_{\ab} \to \pi_1(\Alb(V))$ is an isomorphism, so there exists a subgroup $M \inc \pi_1(\Alb(V))$ of finite index such that $N$ is the preimage of $M$ under $\pi_1(Z_U) \to \pi_1(\Alb(V))$. This finite index subgroup corresponds to an étale cover $\Alb(V)' \to \Alb(V)$ and the fact that $N$ is the preimage of $M$ gives an isomorphism $\overline{X}_U \cong Z_U \times_{\Alb(V)} \Alb(V)'$ over $A(V)$. In particular, the morphism $\overline{X}_U \to Z_U$ can be extended to a finite morphism $A(V)' \to A(V)$. Since it can also be extended to a finite morphism $\overline{g}^{-1}(A(V)) \to A(V)$ and all schemes involved are normal, we deduce that $\overline{g}^{-1}(A(V))$ is isomorphic to $A(V)'$ over $A(V)$ (there are both the normalization of $A(V)$ in the same function field). We have therefore proven that $\overline{g}$ is étale over $A(V)$, so the proof is complete. \qedhere

    %a finite étale cover of $A(V)$. However there is a unique such extension (all schemes involved are normal), namely the integral closure of $A(V)$ in $\bC(\overline{X}_U)$, which is precisely $\overline{g}^{-1}(A(V)) \to A(V)$. In other words, we have proven that $\overline{g}^{-1}(A(V)) \to A(V)$ is étale, concluding the proof. \qedhere

    %\JB{I don't like how what follows is written...} In particular $\bC(\overline{X}) \cong \bC(\Alb(V)')$. Remark that $\Alb(V)'$ is equal to the integral closure of $\cO_{\Alb(V)}$ in $\bC(\Alb(V)')$, and similarly $\overline{g}^{-1}(\Alb(V))$ is equal to the normalization of $\cO_{\Alb(V)}$ in $\bC(\overline{X})$. Since $\Alb(V)' \to \Alb(V)$ is étale, we deduce that $\overline{g}^{-1}(\Alb(V)) \to \Alb(V)$ is also étale, thereby finishing concluding the proof. 
\end{proof}

\begin{thm}[{\autoref{main_thm_max_alb_dim}}]\label{WWPB_mad}
    The Albanese morphism $V \to \Alb(V)$ is an isomorphism away from a closed subset of codimension at least $2$.
\end{thm}

\begin{proof}
    Throughout this proof, in the case where $r = 0$, it is understood that $H = 0$ and that $Z = \Alb(X)$ (so $\Delta = 0$). We need to show that every component of $D-H$ is contracted by $g$. Let $\Gamma$ be such a component and suppose that $g(\Gamma)$ is a divisor on $Z$. Then $(\Gamma, H|_{\Gamma}) \to (Z, \Delta)$ is generically finite, so $h^0(\Gamma, \omega_{\Gamma}(H|_{\Gamma})) \neq 0$. The same argument as in the proof of \autoref{principal_polarization} gives an exact sequence 
    \[ \begin{tikzcd}
        0 \arrow[rr] &  & \cO_{A(X)} \arrow[rr] &  & a_*\omega_X(H + \Gamma) \arrow[rr] &  & a_*\omega_{\Gamma}(H|_{\Gamma}) \arrow[rr] &  & 0
    \end{tikzcd} \] on $A(X)$. In particular, we have that \[ V^0(A(X), a_*\omega_\Gamma(H|_{\Gamma}))\subseteq V^0(A(X), a_*\omega_X(H+\Gamma))\cup V^1(A(X),\cO_{A(X)}), \] so $0 \in V^0(A(X), a_*\omega_{\Gamma}(H|_{\Gamma}))$ is an isolated point by \autoref{lem:GVmorphism}.\autoref{lem:noZ}. The sheaf $a_*\omega_{\Gamma}(H|_{\Gamma})$ must therefore contain $\cO_{A(X)}$ as a direct summand (use the same argument as in the proof of \autoref{lem:GVmorphism}.\autoref{lem:GVmorphismitem}), so $\Gamma$ dominates $A(X)$. This contradicts \autoref{no_bad_component_dominates} by the same argument as in the first paragraph of the proof of \autoref{principal_polarization} (in the case $r = 0$, the fact that $\Gamma$ does not dominate $A(X)$ is automatic for dimensional reasons).
\end{proof}

\begin{corollary}[{\autoref{main_thm_intro_affine_case}}]
    Let $V$ be an affine smooth variety of maximal Albanese dimension, such that $\overline{P}_2(V) = 1$. Then $V$ is isomorphic to $\bG_m^{\dim(V)}$.
\end{corollary}
\begin{proof}
    We will use the same notations as previously in this section. By \autoref{WWPB_mad}, we know that $g \colon X \to Z$ is birational, and that $D \leq g^*\Delta + E$ for some effective $g$--exceptional divisor $E$. Hence, we have that $\kappa(X, D) \leq \kappa(Z, \Delta) = \overline{q}(V) - q(X)$. Since the complement $V$ of $D$ is affine, one sees that the rational map induced by $|mD|$ is an isomorphism on $V$ for $m \gg 0$, so $\kappa(X, D) = \dim(X)$. We have therefore proven that $q(X) = 0$, so $\Alb(V) = \bG_m^{\dim(V)}$. Since $\alb_V \colon V \to \Alb(V)$ is an isomorphism away from a closed subset of $\Alb(V)$ of codimension at least $2$ by \autoref{WWPB_mad}, it is not hard to see using the $S_2$-condition on $\Alb(V)$ that $a_{V, *}\cO_V = \cO_{\Alb(V)}$ (this is a general fact about such morphisms). Given that both $V$ and $\Alb(V)$ are affine, the result follows.
\end{proof}

\section{Compact part of relative dimension 1}\label{sec:largeq}

Our objective in this section is to prove \autoref{main_thm}. Our main strategy is to show that a smooth quasi-projective variety $V$ satisfying the assumption of \autoref{main_thm} has maximal Albanese dimension. The result then immediately follows from \autoref{main_thm_max_alb_dim}. The proof has two key points:
\begin{enumerate}
    \item we show first that, if the Albanese morphism of a good compactification $(X, D)$ of $V$ is an algebraic fiber space, then $a_V$ is dominant and $V$ has maximal Albanese dimension;
    \item successively we show that the Albanese morphism $a \colon  X\rightarrow \Alb(X)$ is an algebraic fiber space.
\end{enumerate}
The first step is  \autoref{cor:best_dominance_we_can_show_so_far}, which is framed in a more general set-up: we do not require that $\overline{q}(V)-q(X)\leq 1$, but we show in general that, if the Albanese morphism of $X$ is an algebraic fiber space with positive dimensional fibers then $\dim(a_V(V))>q(X)$. We hope that this generality will allow us, in the future, to give effective conditions on the plurigenera for a variety to be of maximal Albanese dimension even when $\overline{q}(V)-q(X)\geq 2$. The second step is \autoref{thm:fiberspace}, whose proof is divided in several lemmas after its statement.

We will first need the following general vanishing result:

\begin{proposition}\label{prop:general_kollar_vanishing}
    Let $f \colon X \to Y$ be a morphism of normal projective varieties, with $X$ smooth. Let $B$ be a $\bQ$--divisor on $X$ with coefficients in $[0, 1]$ that has snc support, and let $L$ be a Cartier divisor on $X$ such that $L - (K_X + B) \sim_{\bQ} f^*M$ with $M$ an effective $\bQ$--divisor on $Y$ which is not trivial. Then \[ H^g(Y, f_*\cO_X(L)) = 0, \] where $g \coloneqq \dim(Y)$.
\end{proposition}
\begin{proof}
    If $f$ is not surjective, there is nothing to prove. Therefore we may assume that $f$ is surjective, and we may in fact further assume that $f_*\cO_X = \cO_Y$. We will show the result by induction on $g = \dim(Y)$.

    If $g = 1$, then $M$ is automatically ample, so the result follows from \cite[Theorem 6.3.2]{Fujino_Fundamental_theorems_for_the_log_MMP}. For $g \geq 2$, take a general and ample enough hypersurface section $H$ on $Y$, and set $D \coloneqq f^*H$. Since $f_*\cO_X = \cO_Y$, $D$ is a general member of a semi--ample linear system on $X$, so it is smooth. The restriction $B_D$ of $B$ on $D$ still has snc support, so if we set $L' \coloneqq (L + D)|_D$, we obtain by adjunction that $L' - (K_D + B_D) \sim_{\bQ} f^*M_H$, where $M_H \coloneqq M|_H$ (which is still effective and non--trivial). Consider the exact sequence 
     \[ \begin{tikzcd}
        0 \arrow[rr] &  & \cO_X(L) \arrow[rr] &  & \cO_X(L + D) \arrow[rr] &  & \cO_D(L') \arrow[rr] &  & 0.
    \end{tikzcd} \]

    Then pushing forward to $Y$ gives 
    \[ \begin{tikzcd}
        0 \arrow[rr] &  & f_*\cO_X(L) \arrow[rr] &  & f_*\cO_X(L)\otimes\cO_Y(H) \arrow[rr] &  & f_*\cO_D(L') \connecting{dllll}{""} \\
        && R^1f_*\cO_X(L) \arrow[rr] & & R^1f_*\cO_X(L) \otimes \cO_Y(H) \arrow[rr] && \dots
    \end{tikzcd} \] Since $H$ is general and therefore avoids the associated points of $R^1f_*\cO_X(L)$, we obtain that $R^1f_*\cO_X(L) \to R^1f_*\cO_X(L) \otimes \cO_Y(H)$ is injective (this map corresponds locally to multiplying by an equation cutting out $H$). We then deduce that $H^g(Y, f_*\cO_X(L)) = 0$ by induction and Serre vanishing.
\end{proof}

Throughout, fix a smooth variety $V$ with good compactification $(X, D)$ such that $\overline{P_1}(V) = \overline{P_2}(V) = 1$. Set $g \coloneqq \dim(A(X))$. If $\overline{q}(V)=q(X)$, then it follows from \autoref{lem:GVmorphism}.\autoref{lem:GVmorphismitem} that $V$ (and even $X$) has maximal Albanese dimension, so let us assume instead that $\dim(A(V)) > \dim(A(X))$. Let $Z$ be the standard compactification of $A(V)$, so that the induced map $\pi \colon Z \to A(X)$ is a $(\bP^1)^r$--bundle. Up to blowing up $X$ away from $V$, we are in the situation of \autoref{eq:setupdiagram}.
\begin{comment}
     \begin{equation}\label{eq:setupdiag}
         \begin{tikzcd}
    V \arrow[r, hook] \arrow[d, "a_V"] & X \arrow[d, "g"'] \arrow[dd, "a", bend left = 49] \\
    A(V) \arrow[r, hook]        & Z \arrow[d, "a_Z"']                          \\
                            & A(X).            
    \end{tikzcd}
\end{equation}
\end{comment}
    
    As usual, we set $\Delta \coloneqq Z \setminus A(V)$. We fix a way to write $a_Z \colon Z \to A(X)$ as a sequence of $\bP^1$--bundles $Z \to \dots \to Z_1 \to A(X)$, we let $\Delta_1$ denote the natural boundary of the $\bP^1$--bundle $Z_1 \to A(X)$ induced by $\Delta$, and we let $f_1 \colon X \to Z_1$ denote the induced morphism.

\begin{lemma}\label{lem:image_of_X_avoids_Delta}
    With the set-up above, assume that $f_1$ is not surjective. Then $f_1(X)$ does not intersect $\Delta_1$. 
\end{lemma}
\begin{proof}
    Assume by contradiction that $f_1$ is not surjective and that $f_1(X) \cap \Delta_1 \neq \emptyset$. Let $Y$ denote the normalization of $f_1(X)$, and let $M \coloneqq \Delta_1|_Y$.  Note that $Y \not\subseteq \Delta_1$, since $A(V)$ maps to the complement of $\Delta_1$. In particular, $M$ is an effective divisor of $Y$, and it is non-trivial since $f_1(X) \cap \Delta_1 \neq \emptyset$. Since $f_1^{-1}(M) \inc D$ by construction, there exists a rational number $\varepsilon > 0$ such that $f_1^*(\varepsilon M) \leq D$. Set $B \coloneqq D - \varepsilon f_1^*(M)$, so that $K_X + D \sim_{\bQ} K_X + B + \varepsilon f_1^*(M)$. We can therefore appeal to \autoref{prop:general_kollar_vanishing} to obtain that \[ H^g(Y, f_{1, *}\omega_X(D)) = 0 \] (recall that since $f_1$ is not surjective by assumption, $\dim(f_1(X)) = \dim(A(X)) = g$). Note that since $\Delta_1$ is $\pi_1$--ample and $Y \to Z_1$ is finite, $M$ is also $\pi_1$--ample. We therefore obtain by \cite[Theorem 6.3.(ii)]{Fujino_Fundamental_theorems_for_the_log_MMP} that \[ R^i\pi_{1, *}(f_{1, *}\omega_X(D)) = 0, \] for all $i > 0$, so the Leray spectral sequence gives us that \[ H^g(A(X), a_*\omega_X(D)) = 0. \]
    This is a contradiction, by \autoref{lem:GVmorphism}.\autoref{lem:GVmorphismitem}.
\end{proof}
    
\begin{proposition}\label{cor:best_dominance_we_can_show_so_far}
    Assume that $a_*\cO_X = \cO_{A(X)}$. Then $f_1 \colon X \to Z_1$ is surjective. In particular, $\dim(f(X)) > \dim(A(X))$. 
\end{proposition}
\begin{proof}
    Assume by contradiction that $f_1$ is not surjective, and let $Y$ be a resolution of $f_1(X)$. Let $\bG_m^r \coloneqq \ker(A(V) \to A(X))$, so that $Z_1 \setminus \Delta_1 = A(V)/\bG_m^{r - 1}$ by construction (recall that $r \geq 1$ by hypothesis). Since we are assuming that $f_1$ is not surjective, \autoref{lem:image_of_X_avoids_Delta} ensures that the morphism $Y\rightarrow Z_1$ factors through $A(V)/\bG_m^{r-1}$. Then up to blowing up $X$, we have  the following diagram:
    \[ \begin{tikzcd}
        V \arrow[dd] \arrow[rrrr] &  &                                    &                        & X \arrow[dd, "a"] \arrow[ldd, "f_1"'] \arrow[lld, "g"'] \\
        &  & Y \arrow[dr] \arrow[d]                       &                      &                  \\
        A(V) \arrow[rr]           &  & A(V)/\bG_m^{r - 1} \arrow[r, hook] & Z_1 \arrow[r, "\pi_1"] & A(X).                   
    \end{tikzcd} \]
    Since $H^1(A(V)/\bG_m^{r - 1}, \bC) \to H^1(A(V), \bC)$ is injective (this follows for example from \autoref{sequence_Lie_groups_split}) and $H^1(A(V), \bC) \to H^1(V, \bC)$ is an isomorphism (see \cite[Lemma 3.11]{Fujino_On_quasi_Albanese_maps}), the composition of these two maps is also injective. We then obtain that the map $H^1(A(V)/\bG_m^{r - 1}, \bC) \to H^1(Y, \bC)$ must be injective too.

    Let us show that $h^1(Y, \bC) < h^1(A(V)/\bG_m^{r - 1}, \bC)$, thereby leading to a contradiction. Let $b \colon Y \to A(X)$ denote the composition $Y \to Z_1 \to A(X)$ above. Since $a_*\cO_X = \cO_{A(X)}$, we must also have that $b_*\cO_Y = \cO_{A(X)}$. Given that $Y$ and $A(X)$ have the same dimension by assumption, we conclude that $b$ is birational. In particular, $H^1(Y, \bC) \cong H^1(A(X), \bC)$, so \[ h^1(Y, \bC) = h^1(A(X), \bC) = h^1(A(V)/\bG_m^{r - 1}, \bC) - 1 \] (use e.g. \autoref{sequence_Lie_groups_split}).
\end{proof}

Note that in \autoref{cor:best_dominance_we_can_show_so_far}, we asked that the Albanese morphism of $X$ was an algebraic fiber space. Thus, in order to apply this result to prove \autoref{main_thm}, we must show this property in our case. This is the content of the following theorem:

\begin{thm}\label{thm:fiberspace}
    Let $X$ be a smooth $n$-dimensional projective variety, with $q(X) = n - 1$. Suppose that there is a reduced snc divisor $D$ on $X$ such that
    \[h^0(X,\omega_X(D))=h^0(X,\omega_X(D)^{\otimes 2})=1.\] 
    Then the Albanese morphism $a \colon X \to \Alb(X)$ of $X$ is an algebraic fiber space.
\end{thm}

\begin{proof}
The proof of this result will consist of several steps. First, it follows from \autoref{lem:GVmorphism}.\autoref{lem:GVmorphismitem} that the Albanese morphism $a \colon X \to A(X)$ is surjective. If $P_1(X)>0$, then our assumptions on the logarithmic plurigenera imply that $P_1(X)=P_2(X)=1$, so the result follows directly from \cite[Theorem 3.1]{Jiang_An_effective_version_of_a_thm_of_Kawamata_on_the_Albanese_map}.

For the rest of the proof, we will therefore assume that $P_1(X)=0$. In particular, we have that $D\neq 0$. Let us write $D=D_v+D_h$, where $D_h$ is the sum of components of $D$ that dominate $A(X)$ through $a$.
\begin{lemma}\label{lem:samelogpg}
    We have that $H^0(X,\omega_X(D_h)) \neq 0$. 
\end{lemma}
\begin{proof}
    Consider the short exact sequence
$$
\begin{tikzcd}
0 \arrow[rr] &  & \omega_X(D_h) \arrow[rr] &  & \omega_X(D) \arrow[rr] &  & \omega_{D_v}(D_h|_{D_v}) \arrow[rr] &  & 0,
\end{tikzcd}
$$
and set $\cT \coloneqq a_*\omega_{D_v}(D_h|_{D_v})$ (note that this sheaf is torsion by construction).  We can then pushforward the sequence above and obtain a long exact sequence
\[ \begin{tikzcd}
    0 \arrow[r] & a_*\omega_X(D_h) \arrow[r,"\phi"]   & a_*\omega_X(D) \arrow[r] & \cT \arrow[r] & R^1a_*\omega_X(D_h)\arrow[r]  & \dots
\end{tikzcd} \]
By \autoref{lem:GVmorphism}.\autoref{lem:GVmorphismitem}, there exists a nonzero morphism $\psi \colon a_*\omega_X(D)\rightarrow \mathcal{O}_{A(X)}$. The composition $\psi\circ\phi$ cannot be zero, as otherwise we would obtain a non--zero morphism $\cT'\rightarrow \mathcal{O}_{A(X)}$ for some submodule $\cT' \inc \cT$ (which is impossible since $\cT$ is torsion). Thus we obtained a non-zero morphism of sheaves $\omega_X(D_h)\rightarrow \mathcal{O}_{A(X)}$. Applying the symmetric Fourier--Mukai transform gives a non--zero morphism $\mathbb{C}(0)\rightarrow \FM_A(a_*\omega_X(D_h))$. This shows that $0 \in \bighat{A(X)}$ is in the support of the sheaf $\FM_A(a_*\omega_X(D_h))$, so \autoref{FM_and_cohom} gives us that $H^0(A(X), a_*\omega_X(D_h)) \neq 0$.
 \end{proof}

Thanks to \autoref{lem:samelogpg}, we have that $h^0(X,\omega_X(D_h)) = h^0(X,\omega_X(D_h)^{\otimes 2}) = 1$, so we reduced to the case where $D=D_h$. Now, we consider a smooth model of the Stein factorization of $a$, that is, we have the following diagram

\[ \begin{tikzcd}
X \arrow[rr, "\pi"] \arrow[rrd, "a"'] &  & Y \arrow[d, "b"] \\
 &  & A(X)             
\end{tikzcd} \] with $Y$ smooth, $\pi_*\cO_X = \cO_Y$ and $b$ generically finite. Write $D=\sum_{i \in I} D_i$, and consider the adjoint sequence associated to a component $D_i$ of $D$:
\[ \begin{tikzcd}
    0 \arrow[rr] &  & \omega_X \arrow[rr] &  & \omega_X(D_i) \arrow[rr] & & \omega_{D_i} \arrow[rr] & & 0.
\end{tikzcd} \]
Pushing forward to $Y$ gives a long exact sequence
 \begin{equation}\label{longRpi}
     \begin{tikzcd}
        0 \arrow[r] & \pi_*\omega_X \arrow[r] & \pi_*\omega_X(D_i) \arrow[r] & \pi_*\omega_{D_i} \arrow[r, "\theta_i"] & R^1\pi_*\omega_X \arrow[r] & R^1\pi_*\omega_X(D_i)\rightarrow\cdots
    \end{tikzcd}
\end{equation}
 We also set $\cR_i \coloneqq \ker(\theta_i) \inc \pi_*\omega_{D_i}$. We will need the following generalization of Grauert--Riemenschneider vanishing:

%\begin{thm}[{\cite[Théorème 3.1.(A)]{Esnault_Viewheg_Revetements_cycliques_II}}]\label{generalization_GR_van}
%    Let $f \colon X \to Y$ be a surjective morphism of projective varieties with $X$ smooth, and let $g \colon Y \to Z$ be a generically finite morphism. Then for all $j > 0$ and $i \geq 0$, \[ R^jg_*(R^if_*\omega_X) = 0. \]
%\end{thm}

\begin{lemma}\label{splitting_time}
    Let $i \in I$. The following statements hold:
    \begin{enumerate}
    \item\label{Kollar_torsion_freeness} there exists an isomorphism $R^1\pi_*\omega_X \cong \omega_Y$;
    \item\label{split} the morphism $\theta_i$ from \autoref{longRpi} is split surjective;
    \item\label{vanishR1} for all $j > 0$, we have the vanishings $R^j\pi_*\omega_X(D_i) = 0$ and $R^ja_*\omega_X(D_i) = 0$.
    \end{enumerate}
\end{lemma}
\begin{proof}
    The existence of an isomorphism $R^1\pi_*\omega_X \cong \omega_Y$ follows from \autoref{hodge_module_lemma} (one could also use Koll\'ar's torsion--freeness theorem, see \cite{Kollar_Higher_direct_images_of_dualizing_sheaves_I}).  

    Let us now prove \autoref{split}. First, note that $R^1\pi_*\omega_X(D_i)$ is torsion (this group vanishes at the generic point of $Y$ by Kodaira vanishing for curves), so $\theta_i$ is non-zero. By \autoref{Kollar_torsion_freeness}, it is enough to show that any non-zero morphism $\pi_*\omega_{D_i} \to \omega_Y$ splits. Since
    \[ \Hom(\pi_*\omega_{D_i}, \omega_Y) \cong \Hom(R\pi_*\omega_{D_i}, \omega_Y) \cong \Hom(\omega_{D_i}, \pi^!\omega_Y) \cong \Hom(\omega_{D_i}, \omega_{D_i}) \cong \bC, \]
    any two non-zero morphisms are equal up to multiplication by a non-zero constant. Since the usual trace map $\pi_*\omega_{D_i} \to \omega_Y$ splits (a splitting is given by pulling back top differential forms and dividing by the degree), we deduce that $\theta_i$ also splits.

    Let us finally prove \autoref{vanishR1}. Since $R^j\pi_*\omega_X = 0$ for all $j > 1$ and $R^j\pi_*\omega_{D_i} = 0$ for all $j > 0$ by \autoref{hodge_module_lemma}, the surjectivity of $\theta_i$ gives the vanishing $R^j\pi_*\omega_X(D_i) = 0$ for all $j > 0$. It then follows from the Leray spectral sequence that for all $j > 0$, we have $R^ja_*\omega_X(D_i) = R^jb_*(\pi_*\omega_X(D_i))$. Since we have an exact sequence 
    \[ \begin{tikzcd}
        0 \arrow[rr] &  & \pi_*\omega_X \arrow[rr] &  & \pi_*\omega_X(D_i) \arrow[rr] &  & \cR_i \arrow[rr] &  & 0
    \end{tikzcd} \] and $\cR_i$ is a direct summand of $\pi_*\omega_{D_i}$, we deduce again by \autoref{hodge_module_lemma} that \[  R^jb_*(\cR_i) \inc R^jb_*(\pi_*\omega_{D_i}) \cong R^ja_*\omega_{D_i} = 0 \] for all $j > 0$, so \autoref{vanishR1} is proven.
\end{proof}
\begin{remark}
    Although we showed by hand that any non--zero morphism $\pi_*\omega_{D_i} \to \omega_Y$ splits, we could have invoked the general result \cite[Theorem 4.5]{Hacon_Popa_Schnell_Algebraic_fiber_spaces_over_abelian_varieties:_around_a_recent_theorem___by_Cao_and_Paun}.
\end{remark}
    
The proof of \autoref{thm:fiberspace} will consist in analyzing separately the case in which $D$ has one or at least two components. In both cases, we will need the following result:

\begin{thm}[{\cite[Lemma 4.2 and Theorem 3.2]{Pareschi_Basic_results_on_irr_vars_via_FM_methods}}]\label{Pareschi_result}
    Let $Y$ be a smooth projective variety of maximal Albanese morphism, and let us denote by $b \colon Y \to \Alb(Y)$ its Albanese morphism. If there exists no positive-dimensional closed subset $Z \inc \Pic^0(\Alb(Y))$ such that $Z$ and $-Z$ are contained in $V^0(\Alb(Y), b_*\omega_Y)$, then $b$ is birational. 
\end{thm}

Let us start with the case of two components, because the argument is simpler.

\begin{lemma}\label{2_comps_ok}
    If $D$ has at least two irreducible components, then $a$ is an algebraic fiber space.
\end{lemma}

\begin{proof}
    Let $D_1$ and $D_2$ be two distinct irreducible components of $D$. Pushing forward to $A(X)$ the adjoint sequence
    \[\begin{tikzcd} 
        0 \arrow[rr] &  & \omega_X(D_1) \arrow[rr] &  & \omega_X(D) \arrow[rr] &  & \omega_{D_2}(D_1|_{D_2}) \arrow[rr] &  & 0
    \end{tikzcd} \] and using \autoref{splitting_time}.\autoref{vanishR1} gives the the short exact sequence
    \[ \begin{tikzcd}
         0 \arrow[rr] &  & a_*\omega_X(D_1) \arrow[rr] &  & a_*\omega_X(D) \arrow[rr] &  & a_*\omega_{D_2}(D_1|_{D_2}) \arrow[rr] &  & 0.
    \end{tikzcd} \]
    Given that all the sheaves in this exact sequence are GV--sheaves, we obtain the inclusion $\Supp(\FM_{A(X)}(a_*\omega_{D_2}(D_1|_{D_2}))) \inc \Supp(\FM_{A(X)}(a_*\omega_X(D)))$, or equivalently \[V^0(A(X), a_*\omega_{D_2}(D_1|_{D_2})) \inc V^0(A(X), a_*\omega_X(D)).\] Since $a_*\omega_{D_2}$ contains $b_*\omega_Y$ as a direct summand by points \autoref{Kollar_torsion_freeness} and \autoref{split} of \autoref{splitting_time}, we deduce that \[ V^0(A(X), b_*\omega_Y) \inc V^0(A(X), a_*\omega_{D_2}) \inc V^0(A(X), a_*\omega_{D_2}(D_1|_{D_2})) \inc V^0(A(X), a_*\omega_X(D)). \] Recall that by \autoref{lem:GVmorphism}.\autoref{lem:noZ}, there cannot exist a positive-dimensional closed subset $Z \inc \Pic^0(A(X))$ such that both $Z$ and $-Z$ are included in $V^0(A(X), b_*\omega_Y)$. We then conclude by \autoref{Pareschi_result} that $b \colon Y \to A(X)$ is birational.
\end{proof}

We are left to deal with the case where $D$ is irreducible.

\begin{lemma}\label{irred_ok}
    If $D$ is irreducible, then $a$ is an algebraic fiber space.
\end{lemma}

\begin{proof}
    Let us start by constructing a non-zero morphism $\phi \colon \pi_*\omega_X(D) \to \omega_Y$. By \autoref{lem:GVmorphism}.\autoref{lem:GVmorphismitem}, we know that $a_*\omega_X(D)$ admits a non--zero morphism to $\cO_A = \omega_A$. Since $a_*\omega_X(D) = Ra_*\omega_X(D) = Rb_*R\pi_*\omega_X(D)$ by \autoref{splitting_time}.\autoref{vanishR1}, the non--zero morphism $Rb_*R\pi_*\omega_X(D) \to \omega_A$ corresponds by adjunction to a non--zero morphism $R\pi_*\omega_X(D) \to b^!\omega_A \cong \omega_Y$. We then obtain a non--zero morphism $\phi \colon \pi_*\omega_X(D) \to \omega_Y$ by \autoref{splitting_time}.\autoref{vanishR1}. Now, consider the diagram 
    \[ \begin{tikzcd}
        0 \arrow[rr] &  & \pi_*\omega_X \arrow[rr, "\iota"] &  & \pi_*\omega_X(D) \arrow[rr] \arrow[d, "\phi"] &  & \cR \arrow[rr] &  & 0 \\
        &  & &  & \omega_Y.  &  &  &  &  
    \end{tikzcd} \] Note that the composition $\phi \circ \iota \colon \pi_*\omega_X \to \omega_Y$ must be zero, as otherwise it would be a split surjection by \cite[Theorem 4.5]{Hacon_Popa_Schnell_Algebraic_fiber_spaces_over_abelian_varieties:_around_a_recent_theorem___by_Cao_and_Paun}, contradicting the assumption that $H^0(X, \omega_X) = 0$. Hence, we obtained a non--zero morphism $\psi \colon \cR \to \omega_Y$. Recall that there exists a retract $\lambda \colon \pi_*\omega_D \to \cR$ of the inclusion $\cR \inc \pi_*\omega_D$ by \autoref{splitting_time}.\autoref{split}, so $\psi \circ \lambda \colon \pi_*\omega_D \to \omega_Y$ is a non--zero morphism. The same argument as in the proof of \autoref{splitting_time}.\autoref{split} shows that it must automatically split. We fix such a splitting $\mu \colon \omega_Y \to \pi_*\omega_D$. Here is a diagram illustrating our situation 
\[\begin{tikzcd}[column sep=large, row sep=large]
    \cR \arrow[rr, hook] \arrow[rd, "\psi"'] & & \pi_*\omega_D \arrow[ll, "\lambda"', bend right] \arrow[ld, "\psi \circ \lambda", sloped, pos=0.5, yshift=2pt] \\
    & \omega_Y \arrow[ru, "\mu"', bend right=40] & 
\end{tikzcd}
\]
It then follows immediately that $\lambda \circ \mu$ is a splitting of $\psi$, so in particular $b_*\omega_Y \inc b_*\cR$. By \cite[Théorème 3.1.(A)]{Esnault_Viewheg_Revetements_cycliques_II}, we have that $R^1b_*(\pi_*\omega_X) = 0$, so we have an exact sequence \[ \begin{tikzcd}
        0 \arrow[rr] &  & a_*\omega_X \arrow[rr] &  & a_*\omega_X(D) \arrow[rr] &  & b_*\cR \arrow[rr] &  & 0.
    \end{tikzcd} \] Hence, we obtain as in the proof of \autoref{2_comps_ok} that \[ V^0(A(X), b_*\omega_Y) \inc V^0(A(X), b_*\cR)) \inc V^0(A(X), a_*\omega_X(D)), \] so the same argument as in the proof of \emph{loc. cit.} using \autoref{Pareschi_result} allows us to deduce that $b \colon Y \to A(X)$ is birational.
\end{proof}

The proof of \autoref{thm:fiberspace} is now complete.
\end{proof}

\begin{corollary}[{\autoref{main_thm}}]\label{cor:dominant_when_compact_part__not_too_small}
    Assume that $q(X) \geq \dim(V) - 1$. Then the Albanese morphism $V \to A(V)$ is an isomorphism away from a closed subset of $\Alb(V)$ of codimension at least $2$.
\end{corollary}
\begin{proof}
    Let us show that in this case, $V$ has maximal Albanese dimension, so that the result will follow from \autoref{WWPB_mad}. If $q(X) = \dim(V)$, then this is immediate since $X \to A(X)$ is surjective by \autoref{lem:GVmorphism}.\autoref{lem:GVmorphismitem}. If $q(X) = \dim(V) - 1$, then this follows from \autoref{cor:best_dominance_we_can_show_so_far} and \autoref{thm:fiberspace}.
\end{proof}

\bibliographystyle{amsalpha} 
\bibliography{bibliography}
\end{document}